\documentclass[a4paper]{article}

\usepackage[english]{babel}
\usepackage[utf8x]{inputenc}
\usepackage[T1]{fontenc}
\usepackage{graphicx}
\usepackage{caption, subcaption}
\usepackage[title]{appendix}
\usepackage{tikz}
\usepackage{pgfplots}
\usepackage{cutwin}
\graphicspath{ {images/} }
\usepackage{enumerate}
\usepackage{subcaption}

\usepackage[a4paper,top=2.8cm,bottom=3cm,left=2.7cm,right=2.7cm,marginparwidth=1.75cm]{geometry}

\usepackage{amsmath,yhmath,amsfonts}
\usepackage{amsthm}
\usepackage{amssymb}
\usepackage{dsfont}
\usepackage[hidelinks]{hyperref}
\usepackage{mathtools}
\usepackage{cleveref}
\usepackage{epstopdf}
\usepackage{float}
\usepackage[sort&compress,numbers]{natbib}
\newtheorem{theorem}{Theorem}% 	 	[section]
\newtheorem{lemma}[theorem]{Lemma}%[section]
\theoremstyle{definition}
\newtheorem{definition}[theorem]{Definition}
\newtheorem{remark}[theorem]{Remark}%[section]
\numberwithin{equation}{section}
\numberwithin{theorem}{section}

\newcommand{\de}{\,\mathrm{d}}
\newcommand{\tr}{\mathrm{tr}}

\newcommand{\ess}{\mathrm{ess}\,}

\title{Spectral theoretic characterisation of Markov chain convergence}
\author{Bryn Davies \thanks{Warwick Mathematics Institute, University of Warwick, Coventry CV4 7AL, UK (bryn.davies@warwick.ac.uk) } \thanks{\footnotesize Department of Mathematics, Imperial College London, London SW7 2AZ, UK} \and Angelica Yu Xiao\footnotemark[2] \thanks{Mathematical Institute, University of Oxford, Oxford OX2 6GG, UK} }
\date{}

\begin{document}
	\maketitle

	\begin{abstract}
	In this work, we characterise the statistics of Markov chains by constructing an associated sequence of periodic differential operators. Studying the density of states of these operators reveals the absolutely continuous invariant measure of the Markov chain. This approach also leads to a direct proof of $L^1$ convergence to the invariant measure, along with explicit convergence rates. We show how our method can be applied to a class of related Markov chains including the logistic map, the tent map and Chebyshev maps of arbitrary order.
	\end{abstract}
\vspace{0.5cm}
	\noindent{\textbf{Mathematics Subject Classification (MSC2020):} 37A05, 34L05, 58J51

\vspace{0.2cm}

\vspace{0.5cm}

\section{Introduction}

A common strategy for making mathematical breakthroughs is to exploit connections between apparently unrelated fields. For example, to the untrained eye, dynamical systems and spectral theory for differential operators might seem to be distinct topics, however exploiting connections between the two worlds has been a longstanding basis for breakthroughs. For example, much of the spectral theory of one-dimensional quasi-periodic operators is based on dynamical system formulations of problems (allowing for the language of cocycles and Lyapunov exponents to be used) \cite{avila2009ten, simon1982almost, surace1990schrodinger}. Similarly, moment problems of recurrence relations can be reformulated as spectral problems of self-adjoint operators on suitably chosen Hilbert spaces  \cite{elwahbi2004recursive}. 

In this paper, we will study the behaviour of chaotic nonlinear recursion relations by constructing associated sequences of periodic differential operators whose discriminants satisfy the given relation. We will demonstrate this approach on the canonical example of the logistic map, as well as some related, more general models. This approach is inspired by the seminal work of \cite{kohmoto1983localization}, who studied the spectra of periodic operators with unit cells generated by a Fibonacci tiling rule by deriving a recursion relation satisfied by the associated discriminants. This discovery has underpinned the many subsequent studies of Fibonacci-generated materials \cite{morini2018waves, davies2022symmetry, davies2023super, kolavr1993new, jagannathan2021fibonacci, gumbs1988dynamical}. While this connection has mostly been exploited to use dynamical systems theory to prove statements about the spectra of the differential operators, the converse association is also useful. For example, the statistics of chaotic dynamical systems can be challenging to understand, while the densities of states of periodic elliptic operators are well understood and usually straightforward to characterise explicitly. Recently, this connection was used to characterise the limiting statistics of the chaotic dynamical system associated to Fibonacci-generated materials \cite{davies2024vanhove}, an idea that is generalised in this work. 

The construction employed in this paper uses one-dimensional periodic Hill operators. We consider the monodromy matrices associated to exponentially inflating unit cells $[0,2^n]$, $n\in\mathbb{N}$, which are powers of the first monodromy matrix. As a result, the discriminants (defined as the traces of the monodromy matrices) will follow certain recursion relations. Using the well-known transformation that relates the logistic map with the Mandelbrot set (see \emph{e.g.} \cite{peitgen1986beauty} and references therein), we will show that the discriminants satisfy the logistic map with $r=4$. The exact construction of a periodic operator used here was first observed in \cite[Remark~4.1]{davies2022symmetry}. Whereas \cite{davies2024vanhove} needed a sequence of distinct operators (generated by a Fibonacci tiling rule) to obtain the desired recursion relation, just a single operator is needed here, with artificially inflated unit cells.

There are several benefits of studying the sequence of differential operators instead of the recurrence relations directly. Firstly, we will use the Floquet-Bloch theory for periodic operators to prove that the statistics of the recursion relation converge to the invariant measure (in Section~\ref{sec:convergence}). This direct proof is achieved naturally by visualising the recursion via spectral band diagrams; the crucial observation is that doubling the length of the unit cells folds the spectral bands in half. This approach is not only intuitive but also gives the explicit convergence rates. Additionally, this approach provides new intuition on the features of the invariant measure, for example by linking its poles to the Van Hove singularities that are well-known to appear in the density of states of periodic operators \cite{van1953occurrence}.

We will extend these ideas to generalised logistic maps in Section \ref{sec:generalised_logistic}, where we inflate the unit cell by an arbitrary ratio $m\in\mathbb{N}$. This is a natural generalisation of the arguments for the logistic map, so we obtain analogous convergence results with corresponding convergence rate estimates. In Section~\ref{sec:explicit_formulas}, we show that our approach can be used to generate explicit formulas for orbits of the recursion relations. This is achieved by choosing suitable potentials in the Hill operator, such that it is possible to solve the differential equations and get explicit expressions for the discriminants. For example, taking homogeneous potentials leads to the well-known formula for logistic map solutions $x_n=\sin^2(2^n \sin^{-1}(\sqrt{x_0}))$ from \cite{lorenz1964problem}. We also found a different formula that involves Mathieu functions, by choosing sinusoidal potentials. Finally, in Section \ref{sec:conjugacy}, we will explore the conjugacy of the generalised logistic maps to the Chebyshev polynomials and multiple-tent maps. This widens the class of Markov chains for which our method can be applied and also provides an avenue to understanding their Lyapunov exponents, in Section~\ref{sec:lyapunov}.

\section{Mathematical prerequisites}

\subsection{The logistic map}

To demonstrate our approach, we will consider the canonical example of the logistic map. The logistic map was popularised by Robert May in 1976 \cite{may1976simple} and since then become a stereotypical introductory example of a chaotic dynamical system. It is the recursion relation
\begin{equation} \label{eq:logistic_map_def}
    x_{n+1} = F(x_n) \coloneqq rx_n(1-x_n),
\end{equation}
where $r>0$ is a real-valued parameter. The logistic map is often used as a simple model for population dynamics in discrete time. The normalised $x_n\in[0,1]$ represents the population at time $n$ as a proportion of the maximum possible population accommodated by limited resources. A future population $x_{n+1}$ depends on the current population $x_n$, the remaining resources $(1-x_n)$ and the growth rate $r>0$.

The qualitative behaviour of the logistic map \eqref{eq:logistic_map_def} depends on the growth rate parameter $r$. For any $r$ in $[0,4]$, the function \eqref{eq:logistic_map_def} maps $[0,1]$ to $[0,1]$, so that the values remain bounded. However, different behaviours of the system are observed for different values of $r$.  If $0<r<1$, then the population will eventually die out (\emph{i.e.} $x_n\to0$). If $1<r<3$, then the population will converge to the stable fixed point $\frac{r-1}{r}$. For $3<r<1+\sqrt{6}$, a period doubling bifurcation occurs, and the system will converge to a stable 2-periodic limit cycle. The most interesting behaviour happens when $1+\sqrt{6}<r\leq 4$, as the system undergoes successive period doubling bifurcations and becomes chaotic. A review of the rich and complicated behaviour that can happen in the chaotic regime can be found in \emph{e.g.} \cite{lyubich2000quadratic}.

The most widely studied case of the logistic map is the special case $r=4$. In this case, the chaotic behaviour of orbits has been studied extensively. There is even an explicit formula for the solutions:
\begin{equation} \label{eq:explicit_formula_sine}
    x_n=\sin^2(2^n\sin^{-1}(\sqrt{x_0})),
\end{equation}
thanks to \cite{lorenz1964problem}. It is also well known that the logistic map with $r=4$ admits an absolutely continuous invariant measure, given by the density
\begin{equation} \label{eq:logistic_invariant_measure}
    q(x) = \frac{1}{\pi}\frac{1}{\sqrt{x(1-x)}}.
\end{equation}
That is, if the random variable $X_0$ has the probability density function $q(x)$, then $X_n = F^n(X_0)$ will have the same probability density function \eqref{eq:logistic_invariant_measure}, for all $n\in\mathbb{N}$.

In this paper, we focus on the question of having a random variable $X_0$ as the initial value and seeking to characterise the distributions of the sequence of random variables $X_n = F^n(X_0)$. Due to the fact that there is a well-known absolutely continuous invariant measure and the interval $(0,1)$ is an attractor, the invariant measure \eqref{eq:logistic_invariant_measure} is also a limiting measure. That is, starting from any $X_0$ with absolutely continuous measure, the measure of the sequence $(X_n)$ will converge to it. This is formalised in the following theorem, which can be proved using standard results in ergodic theory. Its proof is given in Appendix~\ref{app:ergodic}.

\begin{theorem} \label{thm:ergodic}
For random variable $X_0$ with a probability density function $p_0\in L^1[0,1]$ such that $p_0\geq0$ and $\int_0^1 p_0 = 1$, the measure of $X_n = F^n(X_0)$ converges set-wise to the invariant measure defined by the density function \eqref{eq:logistic_invariant_measure}:
\begin{equation*}
    \lim_{n\rightarrow\infty}P(X_n\in A)=\int_A q(x)\,dx\quad\text{for all measurable sets $A\subseteq[0,1]$.}
\end{equation*}
\end{theorem}

For strong convergence, standard ergodic theory often works with the $BV$ space of bounded variation functions (see e.g. \cite{rychlik1983bounded, bruin2010existence, lasota1973existence}). Thus, the results do not apply directly for the logistic map $F$, whose invariant density \eqref{eq:logistic_invariant_measure} blows up at endpoints. In contrast, this work takes advantage of the reformulation of the problem in terms of the spectra of a sequence of periodic differential operators to give a direct proof of strong convergence in the $L^1$ space. It will turn out, that this direct proof will come with explicit convergence rate estimates and can be generalised to a class of related Markov chains. It will also reveal fundamental mechanisms for the properties of the invariant measure, such as how singularities are due to the underlying Van Hove singularities in the spectra of the periodic operators.

\subsection{Periodic differential operators} \label{sec:periodic_differential_operators}
The differential operator whose spectrum we consider in this paper is the one-dimensional Hill operator 
\begin{equation} \label{eq:Hdef}
    H = -\frac{\de^2}{\de x^2}+V(x).
\end{equation}
where $V$ is an $l$-periodic real-valued potential function. See \cite{kuchment2016overview} for a comprehensive review of the spectral theory of such periodic elliptic differential operators. In this section we will introduce the key ideas that we will need: Floquet-Bloch theory, monodromy matrices, dispersion relations and the density of states, all of which are crucial to our approach.

\subsubsection{Monodromy matrix and discriminant}

\begin{definition}
    Let $\varphi_\lambda$ and $\psi_\lambda$ be solutions of $Hu=\lambda u$ satisfying the initial conditions
    \begin{equation} \label{eq:basis_solutions_bc}
        \varphi_\lambda(0)=\psi_\lambda'(0)=1,\quad \varphi_\lambda'(0)=\psi_\lambda(0)=0.
    \end{equation}
    For an $l$-periodic potential $V$, we define the monodromy matrix for the unit cell $[0,l]$ as
    \begin{equation} \label{eq:MonodMatDefn}
        M_l(\lambda)\coloneq\begin{bmatrix}
    \varphi_\lambda(l) & \psi_\lambda(l) \\
    \varphi_\lambda'(l) & \psi_\lambda'(l) 
    \end{bmatrix}.
    \end{equation}
\end{definition}
\begin{remark} \label{rmk:monodromy_action}
    If $y$ is a solution to $Hy=\lambda y$, then it is a linear combination of $\varphi_\lambda$ and $\psi_\lambda$ and it must hold that
    \begin{equation*}
        \begin{bmatrix}
            y(l) \\
            y'(l)
        \end{bmatrix} = M_l\begin{bmatrix}
            y(0) \\
            y'(0)
        \end{bmatrix}
    \end{equation*}
\end{remark}

Since the operator $H$ \eqref{eq:Hdef} does not have a first order term, $\det M_l$ the Wronskian of the Hill operator is constant. This can be interpreted physically as there not being any damping in the system, so the total energy in the system (captured by the Wronskian) is conserved. In particular, it holds that
\begin{equation} \label{eq:det=1}
    \det M_l\equiv 1\quad\text{for all $l$.}
\end{equation}
Then the monodromy matrix $M_l$ can be characterised by its trace, which fully determines its eigenvalues.

\begin{definition}
    The discriminant $\Delta_l$ for the Hill operator is defined as the trace of the monodromy matrix:
    \begin{equation*}
    \Delta_l \coloneqq \tr M_l.
\end{equation*}
\end{definition}

\begin{lemma} \label{lm:discriminant}
The eigenvalues of the monodromy matrix $M_l$ are determined by the value of the discriminant $\Delta_l$. Further,
    \begin{enumerate}
        \item When $|\Delta_l|<2$, $M_l$ has two complex conjugate eigenvalues with modulus 1.
        \item When $|\Delta_l|=2$, $M_l$ has a repeated eigenvalue equal to either 1 or -1.
        \item When $|\Delta_l|>2$, $M_l$ has two distinct real eigenvalues, which multiply to one.
    \end{enumerate}
\end{lemma}

\subsubsection{Floquet-Bloch spectrum}

We also want to consider the spectrum of $H$. It can be shown that the spectrum is a union of countably many closed intervals, which are called the spectral bands.
\begin{lemma}
    There exists a sequence of real numbers
    \begin{equation*}
        a_1<b_1\leq a_2<b_2\leq a_3<\dots\rightarrow\infty
    \end{equation*}
    such that the spectrum of the Hill operator is
    \begin{equation*}
        \sigma(H)=\bigcup_{i=1}^\infty[a_i,b_i].
    \end{equation*}
\end{lemma}

The spectrum of $H$ can be decomposed using the Floquet-Bloch theory for periodic operators \cite{kuchment2016overview}. This states that any solution to the eigenvalue problem $Hy = \lambda y$ is of the form 
\begin{equation*}
    y(x)=e^{ikx}p(x),
\end{equation*}
where $p$ is a real $l$-periodic function and $k$ is a real-valued parameter called the quasi-momentum. Moreover, $e^{ikl}$ must be an eigenvalue of the monodromy matrix \eqref{eq:MonodMatDefn}. Conversely, whenever $e^{ikl}$ is an eigenvalue of $M_l(\lambda)$, there exists a solution of $Hy=\lambda y$ in the form $y(x)=e^{ikx}p(x)$. Since $k\in\mathbb{R}$, the solutions are bounded. This occurs exactly when $|\Delta_l(\lambda)|\leq2$, by Lemma \ref{lm:discriminant}. 

\begin{lemma}
    $\lambda\in\sigma(H)$ if and only if there exists $k\in\mathbb{R}$ such that $(k,\lambda)$ satisfies the following dispersion relation:
    \begin{equation} \label{eq:dispersion}
    \Delta_l(\lambda) = 2\cos(lk).
\end{equation}
\end{lemma}

Solving \eqref{eq:dispersion} will give us a multiple-valued map $k\in\mathbb{R}\mapsto\lambda\in\mathbb{R}$ with countably many branches. If we label the branches in increasing order as $\lambda_1(k)\leq\lambda_2(k)\leq\dots$, then the image of the $i$th branch corresponds to the $i$th spectral band $[a_i,b_i]\in\sigma(H)$. A plot of these branches is called a band diagram. Its projection onto the $\lambda$-axis gives the spectrum of $H$. An example is shown in Figure \ref{fig:band_diagram_example}.

Note that \eqref{eq:dispersion} is $\frac{2\pi}{l}$-periodic, and so are the branches. This reduces the range of $k$ to the Brillouin zone $[-\frac{\pi}{l},\frac{\pi}{l}]$. Moreover, as shown in Figure \ref{fig:band_diagram_example}, \eqref{eq:dispersion} is even, so it suffices to consider $k$ in the reduced Brillouin zone $[0,\frac{\pi}{l}]$.

\begin{figure}
    \centering
    \includegraphics[width=0.4\linewidth]{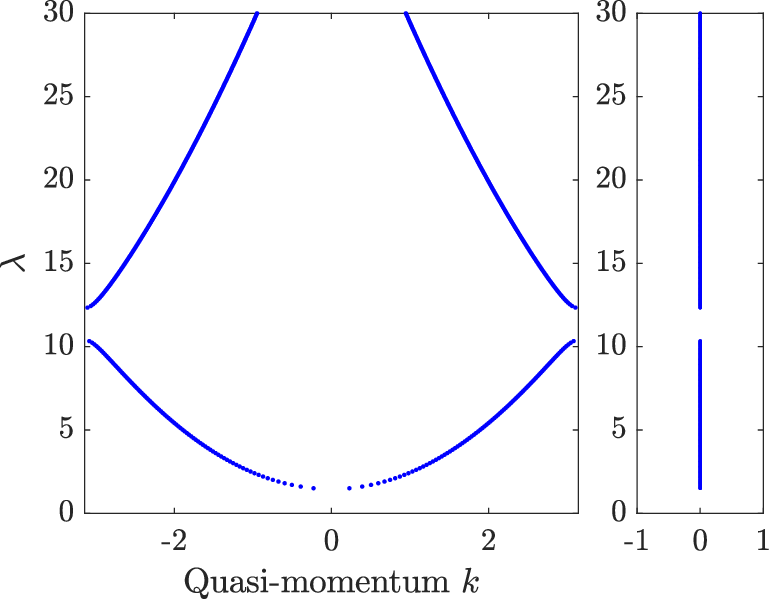}
    \caption{The band diagram of a $1$-periodic piecewise linear potential on the Brillouin zone $[-\pi,\pi]$. Only parts of the first two spectral bands are shown, and countably many more bands exist along the $\lambda$-axis. On the right is the projection on to the $\lambda$-axis, which coincides with the spectrum of the Hill operator $H$.}
    \label{fig:band_diagram_example}
\end{figure}

\subsubsection{Density of states}

The value of reformulating the problem in terms of periodic differential operators is that the associated density of states is well understood, giving us a ready-made tool to keep track of where states are likely to fall within the spectrum.

The density of states of a periodic operator can be defined in terms of the limit of the spectra on finite-sized domains. Consider a sequence of expanding finite-measure domains $A_N\subset\mathbb{R}$ that tend to cover the whole space $\mathbb{R}$ as $N\rightarrow\infty$. On each $A_N$ with finite measure, $H$ has a discrete spectrum with countably many eigenvalues $\lambda_i$. The integrated density of states is defined to be the limit of the counting functions of eigenvalues:
\begin{equation*}
    \rho(\lambda)\coloneqq\lim_{N\rightarrow\infty}\frac{1}{|A_N|}\#\{\lambda_i<\lambda\}.
\end{equation*}
Under appropriate assumptions on the operator and the boundary conditions, this limit is well-defined. For more details see the review \cite{kuchment2016overview} and references therein. The density of states is then calculated as the Radon-Nikodym derivative of $\rho$ with respect to the Lebesgue measure $L$:
\begin{equation*}
    p(\lambda) \coloneqq \frac{\de\rho}{\de L}.
\end{equation*}

Similarly, we could consider the density of the discriminants $\Delta_l$ in $[-2,2]$, which is related to the density of states by the function $\Delta_l(\lambda)$. It was shown in \cite{davies2024vanhove} that this density can be derived by differentiating the dispersion relation \eqref{eq:dispersion} with respect to $k$. This density is shown in Figure~\ref{fig:delta_density} and turns out to be common for all Hill operators, independent of the potential $V$ and the period $l$:
\begin{equation} \label{eq:delta_density}
    D(\Delta)=\frac{1}{\pi}\frac{1}{\sqrt{4-\Delta^2}} \quad\text{for $-2<\Delta<2$.}
\end{equation}

\begin{figure}
    \centering
    \includegraphics[width=0.55\linewidth]{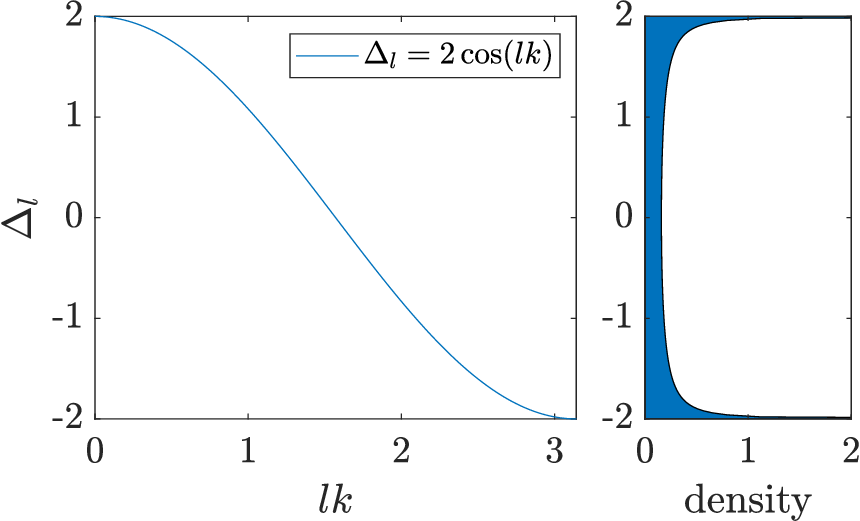}
    \caption{The dispersion relation \eqref{eq:dispersion} shown on the left, relating the discriminant $\Delta_l$ and the normalised quasi-momentum $lk$. The corresponding density of $\Delta_l$, given by \eqref{eq:delta_density}, is shown on the right. In this way, we can think of $\Delta_l$ as a function of $lk$.} 
    \label{fig:delta_density}
\end{figure}

\section{Link between recursion relation and operators} \label{sec:link}

The density of the discriminants \eqref{eq:delta_density} is our main motivation for studying recursion relations via the Hill operator. Since eigenstates of the Hill operator can be related to states of the recursion relation via the discriminant, we can use our knowledge of this density to predict the statistics of the associated nonlinear recursion relation. In Section \ref{sec:candidate_invariance}, we will show how the density of the discriminants \eqref{eq:delta_density} is  related to the invariant measure of the logistic map \eqref{eq:logistic_invariant_measure} through a change of variables. This will reveal how features of the invariant measure are related to those of the density of states of the periodic operators.  For example, the typical concentration of values of the recursion relation near the edges of the interval is related to the Van Hove singularities that occur in the density of states at the edges of the spectral bands \cite{davies2024vanhove, van1953occurrence}. Our method gives a route to transfer this knowledge and intuition to dynamical systems.

To construct the logistic map from the Hill operator, consider any $1$-periodic potential $V$. This $V$ is also $l$-periodic for any $l\in\mathbb{N}$, so we could apply the analysis in Section \ref{sec:periodic_differential_operators} repeatedly for different unit cell length $l$. In particular, consider the sequence $l_n=2^n, n\in\mathbb{N}$ and their monodromy matrices $M_{2^n}$. Recall from Remark~\ref{rmk:monodromy_action} that the monodromy matrix for a segment $[0,l]$ maps the vector $[y(0);y'(0)]$ to $[y(l);y'(l)]$. Since $H$ is 1-periodic,
\begin{equation} \label{eq:monodromy_periodicity}
    M_{k+l}=M_kM_l\quad\text{for all $k,l\in\mathbb{N}$.}
\end{equation}
In particular, we have the following relation
\begin{equation*}
    M_{2^{n+1}} = M_{2^n}^2\quad\text{for all $n\in\mathbb{N}$}.
\end{equation*}

Consider the corresponding sequence of discriminants $\Delta_{2^n}=\tr M_{2^n}$. For a general 2-by-2 matrix $A$, we have the identity
\begin{equation*}
    \tr(A^2) = \tr(A)^2 - 2\det(A).
\end{equation*}
Recall that $\det M_l\equiv1$ from \eqref{eq:det=1}, so we have the following recursion relation
\begin{equation*}
    f_2(\Delta_{2^n})\coloneqq\Delta_{2^{n+1}}=\Delta_{2^n}^2-2\quad\text{for all $n\in\mathbb{N}$.}
\end{equation*}
This is a particular case of the recursion relation $z\rightarrow z^2-c$ that is used to define the Mandelbrot set \cite{mandelbrot2004fractals}, so we apply the well-known conjugacy (change of variables) between the logistic map and the Mandelbrot set:
\begin{equation} \label{eq:logistic_mandelbrot_conjugacy}
    x_n=\frac{1}{4}(2-\Delta_{2^n}),\quad\Delta_{2^n}=2-4x_n.
\end{equation}
Then $(x_n)$ will satisfy the logistic map with $r=4$: $x_{n+1}=4x_n(1-x_n)$ for all $n\in\mathbb{N}$. That is, we have constructed a sequence of Hill operators (with artificially doubled unit cells) that has an associated sequence of discriminants which satisfies the logistic map.

\subsection{Candidate invariant measure} \label{sec:candidate_invariance}

If we apply the change of variables \eqref{eq:logistic_mandelbrot_conjugacy} to the density of states \eqref{eq:delta_density}, then it coincides with the well-known invariant measure of the logistic map with $r=4$ \eqref{eq:logistic_invariant_measure}:
\begin{equation*}
    D(2-4x)\cdot\left|\frac{\de\Delta}{\de x}\right|=\frac{1}{\pi}\frac{1}{\sqrt{4-(2-4x)^2}}\cdot 4=\frac{1}{\pi}\frac{1}{\sqrt{x(1-x)}}=q(x).
\end{equation*}

We propose that for any recursion relation that can be constructed from the discriminants of Hill operators, one should consider the density of states \eqref{eq:delta_density} as a candidate invariant measure. For example, \cite{davies2024vanhove} constructed a 3\textsuperscript{rd} degree recursion relation using Fibonacci tiling potentials. When restricted to the interval $[-2,2]$, this recursion relation does converge to the invariant measure \eqref{eq:delta_density} numerically. In the next section, we will also construct the generalised logistic maps by further exploiting the inflating unit cells. It turns out that they all have the same invariant measure \eqref{eq:delta_density}, corroborating our hypothesis.

\subsection{Generalised logistic maps} \label{sec:generalised_logistic}

Instead of considering the sequence $(l_n=2^n:{n\in\mathbb{N}})$ that doubles at each recursion, we may take any $m\in\mathbb{N}$ and consider the sequence $(l_n=m^n:{n\in\mathbb{N}})$. Again, thanks to the periodicity of $H$ \eqref{eq:monodromy_periodicity}, we have
\begin{equation} \label{eq:MatrixPowers}
    M_{m^{n+1}} = M_{m^n}^m\quad\text{for all $m,n\in\mathbb{N}$}.
\end{equation}
The question is then to find the trace of an integer power of a matrix. The general formula of $\tr(A^m)$ as a function of $\tr (A)$ and $\det (A)$ for 2-by-2 matrices was shown in \cite{pahade2015trace} to be given by
\begin{align*}
    \begin{split}
        \tr(A^m)=\sum_{r=0}^{\lfloor m/2 \rfloor}\frac{(-1)^r}{r!}m(m-(r+1))(m-(r+2))\cdots (\mbox{up to r terms})\cdot \det(A)^r\cdot \tr(A)^{m-2r},
    \end{split}
\end{align*}
for $m\in\mathbb{N}$. Recalling that monodromy matrices always satisfy $\det M_l\equiv1$, we may define a collection of maps which will be satisfied by the discriminants of the exponentiated matrices \eqref{eq:MatrixPowers}. We will call these the \emph{generalised logistic maps} for each degree $m\in\mathbb{N}$, which are given by
\begin{equation} \label{eq:generalised_logistic_def}
    f_m(x) =  \sum_{r=0}^{\lfloor m/2 \rfloor} \frac{(-1)^r}{r!}m(m-(r+1))(m-(r+2))\cdots (\mbox{up to r terms})\cdot x^{m-2r},
\end{equation}
so that the sequence $(\Delta_{m^n}:{n\in\mathbb{N}})$ follows the recursion $f_m$:
\begin{equation*} 
    \Delta_{m^{n+1}} = f_m(\Delta_{m^n})\quad\text{for all $n\in\mathbb{N}$.}
\end{equation*}

The generalised logistic maps $f_m$ are monic polynomials of degree $m$ that map $[-2,2]$ to $[-2,2]$. The first few are shown in Figure \ref{fig:fm} and are given by
\begin{align*}
    f_2(x)&=x^2-2\\
    f_3(x)&=x^3-3x\\
    f_4(x)&=x^4-4x^2+2\\
    f_5(x)&=x^5-5x^3+5x
\end{align*}
It can be verified that for all $m\geq2$, our proposed candidate invariant measure \eqref{eq:delta_density} is indeed the unique absolutely continuous invariant measure of $f_m$.

\begin{figure}
    \centering
    \includegraphics[width=0.5\linewidth]{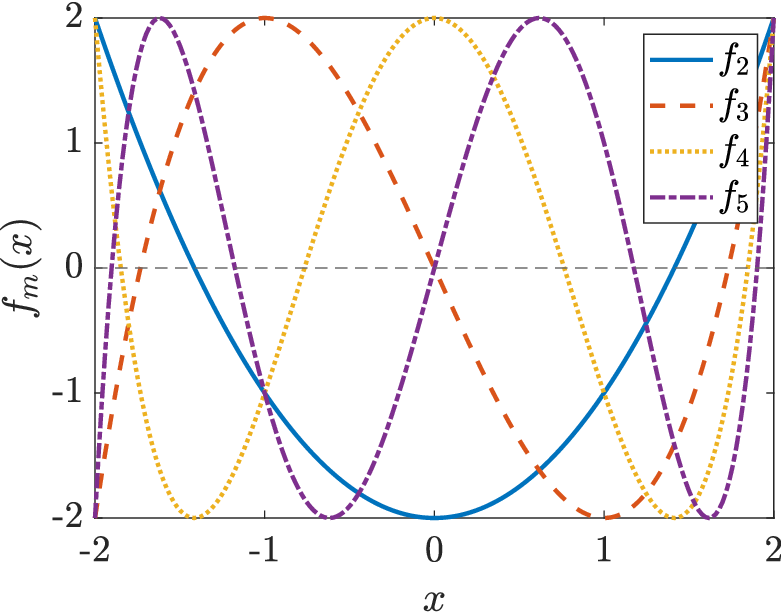}
    \caption{The first few generalised logistic maps. Each $f_m$ is a monic polynomial of degree $m$ that maps the interval $[-2,2]$ to $[-2,2]$. Its $m$ roots are all real and in the interval $(-2,2)$.} 
    \label{fig:fm}
\end{figure}

\subsection{Convergence to invariant measure} \label{sec:convergence}

Reformulating the dynamical system in terms of a sequence of Hill operators is useful not only for quickly generating candidate invariant measures, as discussed above, but can also be used to produce a direct proof of convergence. Our main convergence result for the logistic map is the following:
\begin{theorem} \label{thm:logistic_convergence}
    Suppose $X_0$ have an absolutely continuous probability distribution with density function $p_0\in L^1[0,1]$. Let $p_n$ be the probability density function of $X_n = F^n(X_0)$, where $F$ is the logistic map \eqref{eq:logistic_map_def} with $r=4$. Then
    \begin{equation*}
        p_n \rightarrow q\text{ as $n\rightarrow\infty$}\quad\text{in $ L^1[0,1]$,}
    \end{equation*}
    where $q$ is the invariant density \eqref{eq:logistic_invariant_measure}.

    Moreover, if $p_0\in BV[0,1]$ the space of bounded variation functions on $[0,1]$, then the convergence rate is at least $O(2^{-n})$:
    \begin{equation*}
        \lVert p_n-q\rVert_{ L^1[0,1]}\leq \frac{c}{2^{n}}\quad\text{for all $n\in\mathbb{N}$}, 
    \end{equation*}
    where $c$ is a constant depending on $p_0$.
\end{theorem}

We have the analogous results for the generalised logistic maps, with the domain changed to the interval $[-2,2]$ instead of $[0,1]$.
\begin{theorem} \label{thm:generalised_logistic_convergence}
    Suppose $X_0$ have an absolutely continuous probability distribution with density function $p_0\in L^1[-2,2]$. Let $p_n^m$ be the probability density function of $X_n = f_m^n(X_0)$ where $f_m$ is the generalised logistic map \eqref{eq:generalised_logistic_def}. Then
    \begin{equation*}
        p_n^m \rightarrow D\text{ as $n\rightarrow\infty$}\quad\text{in $ L^1[-2,2]$ for all $m\in\mathbb{N}$,}
    \end{equation*}
    where $D$ is the density of states \eqref{eq:delta_density}.

    Moreover, if $p_0\in BV[-2,2]$ the space of bounded variation functions on $[-2,2]$, then the convergence rate is at least $O(m^{-n})$:
    \begin{equation} \label{eq:exponential_convergence}
        \lVert p_n^m-D\rVert_{ L^1[-2,2]}\leq \frac{\Tilde{c}}{m^{n}}\quad\text{for all $m,n\in\mathbb{N}$,}
    \end{equation}
    where $\Tilde{c}$ is a constant depending only on $p_0$.
\end{theorem}

Our proof of Theorems \ref{thm:logistic_convergence} \& \ref{thm:generalised_logistic_convergence} will be based on the spectral band diagram of the free Hill operator with $V\equiv0$. We briefly survey its properties, before presenting the proof. 

\subsubsection{The free Hill operator}

To plot the diagram, we need to solve the dispersion relation \eqref{eq:dispersion}. We start by calculating $\Delta_l(\lambda)$. For the free operator, the eigenvalue problem $Hy = \lambda y$ is solved by $y(x) = \sin{(\sqrt{\lambda} x + \alpha)}$ when $\lambda\geq0$, for some phase constant $\alpha$. Then the basis solutions $\varphi_\lambda$, $\psi_\lambda$ satisfying the boundary conditions \eqref{eq:basis_solutions_bc} are
\begin{equation*}
    \varphi_\lambda(x)=\cos(\sqrt{\lambda} x)\quad\text{and}\quad \psi_\lambda(x)=\frac{1}{\sqrt{\lambda}}\sin(\sqrt{\lambda} x),
\end{equation*}
so the monodromy matrix is
\begin{equation*}
    M_l(\lambda)=\begin{bmatrix}
\varphi_\lambda(l) & \psi_\lambda(l) \\
\varphi_\lambda'(l) & \psi_\lambda'(l) 
\end{bmatrix}
= \begin{bmatrix}
\cos(l\sqrt{\lambda}) & \frac{1}{\sqrt{\lambda}}\sin(l\sqrt{\lambda}) \\
-\sqrt{\lambda}\sin(l\sqrt{\lambda}) & \cos(l\sqrt{\lambda})
\end{bmatrix},
\end{equation*}
and the discriminant is
\begin{equation*}
    \Delta_l(\lambda)= \tr M_l=2\cos(l\sqrt{\lambda}).
\end{equation*}
Then, the dispersion relation \eqref{eq:dispersion} becomes
\begin{equation} \label{eq:free_dispersion}
    2\cos(l\sqrt{\lambda}) = \Delta_l(\lambda) = 2\cos(lk)\quad\text{for all $l\in\mathbb{N}$.}
\end{equation}
For each $l\in\mathbb{N}$, \eqref{eq:free_dispersion} can be solved by the band functions
\begin{equation} \label{eq:free_band_functions}
    \lambda_i^l(k) = \left(k+\frac{2\pi}{l} i\right)^2\quad\text{for $i\in\mathbb{Z}$.}
\end{equation}

Recall that in our construction, we consider a sequence of expanding unit cells with length $l_n = m^n$. Take $m=2$ for an example. Then \eqref{eq:free_band_functions} gives us a sequence of band diagrams for each $l_n=2^n$, $n\in\mathbb{N}$. Figure~\ref{fig:folding}(a)-(c) shows the band diagrams for $l=1,2,4$. 

\begin{figure}
    \centering
    \includegraphics[width=0.8\linewidth]{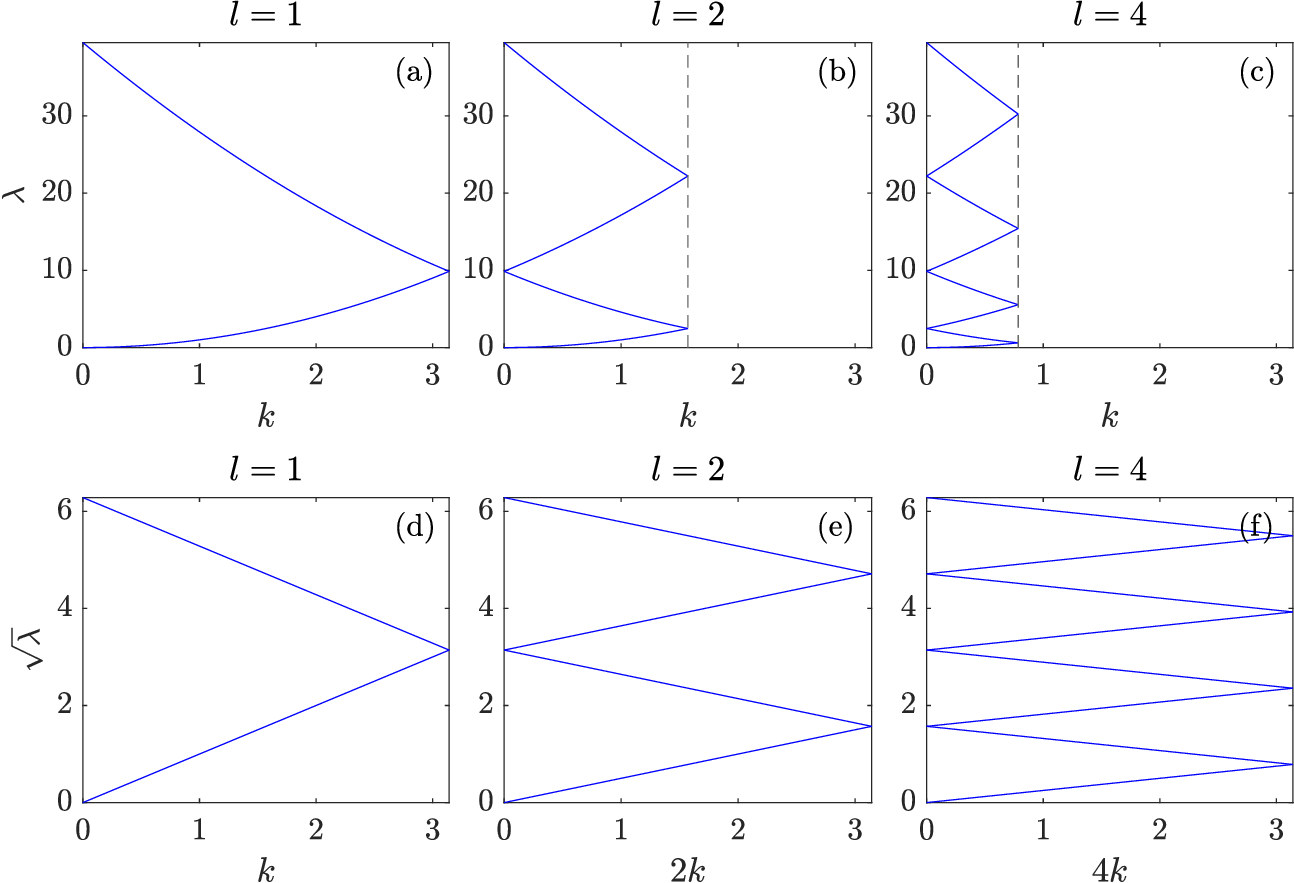}
    \caption{(a)-(c) show the band diagrams of the free Hill operator when $l=1,2,4$, with the quasi-momentum in the respective reduced Brillouin zone. Only the first few bands are shown in each diagram. (d)-(f) show the normalised band diagrams, where the normalised quasi-momentum $lk$ is shown on the $x$-axis and the square root of the eigenvalue on the $y$-axis.}
    \label{fig:folding}
\end{figure}

The crucial idea is that after doubling the length of the unit cell, the band diagram is folded along the middle to fit the halved reduced Brillouin zone $[0,\frac{\pi}{2^n}]$. This can be made more obvious by putting $\sqrt{\lambda}$ in the vertical axis to linearize the bands, as shown in Figure~\ref{fig:folding}(d)-(f), where we also use $lk$ for the horizontal axis to normalise the reduced Brillouin zone to $[0,\pi]$. To formalise this observation, we will use normalised variables $\kappa_l = lk/\pi$ and $\Lambda=\sqrt{\lambda}/\pi$, so that the reduced Brillouin zone is fixed as $[0,1]$. Then, we have a normalised dispersion relation
\begin{equation} \label{eq:normalised_free_dispersion}
    \Delta_l = 2\cos(\pi\kappa_l),
\end{equation}
We also define the linear map shown in Figure~\ref{fig:folding}(d)-(f). Note that for each fixed $l$, there is a well-defined map $\lambda\mapsto k$ that is the inverse of the band functions $\lambda_i(k)$, so also a normalised map $\Lambda\mapsto\kappa_l$. We write this as $K_l:[0,\infty)\rightarrow[0,1]$ such that
\begin{equation*}
    K_l(\Lambda)=\kappa_l.
\end{equation*}
$K_l$ is piecewise linear thanks to the normalisation. Moreover, as we increase $l$, $K_l$ is piecewise linear on finer intervals. This is exactly the result of folding.

\begin{lemma}
    For all $l\in\mathbb{N}$, $K_l$ is linear and bijective on each interval $[\frac{i}{l},\frac{i+1}{l}]$, for $i\in\mathbb{N}$.
\end{lemma}
\begin{proof}
    This is clear from the nature of the curves in Figure~\ref{fig:folding}(d)-(f). Alternatively, substituting \eqref{eq:free_band_functions} in the definition of normalised variables and $K_l$ gives
    \begin{equation} \label{eq:K_l}
        K_l(x)=\begin{cases}
        lx-2i & \text{if $x\in[\frac{2i}{l},\frac{2i+1}{l}]$ for $i=0,1,\dots$} \\
        -lx+2i & \text{if $x\in[\frac{2i-1}{l},\frac{2i}{l}]$ for $i=1,2,\dots$}
    \end{cases}
    \end{equation}
\end{proof}

To conclude, we translated the recursion of discriminants $(\Delta_{l_n}:n\in\mathbb{N})$ into the language of band diagrams. Given an initial value $\Delta_1$, there is an eigenvalue $\lambda$ that is determined uniquely up to the choice of spectral band and satisfies
\begin{equation} \label{eq:Lambda_Delta_relation}
    2\cos(\sqrt{\lambda})=\Delta_1.
\end{equation}
We fix the choice of band by taking $l=1$ in \eqref{eq:free_dispersion}. This eigenvalue $\lambda$ will represent our initial value. The subsequent values $\Delta_{l_n}(\lambda)$ evolve as the band diagrams get folded whenever the unit cell length is increased. This is visualised most easily by considering the normalised quasi-momentum $\kappa_l$ and normalised eigenvalue $\Lambda$, along with the piecewise linear map $K_l$ between them.

\subsubsection{Frobenius-Perron operator}

Our strategy for proving Theorems~\ref{thm:logistic_convergence} and~\ref{thm:generalised_logistic_convergence} is to show that the distribution of $\kappa_l$ will converge to the uniform distribution on $[0,1]$, and then show that this implies the distribution of $\Delta_l$ converges to \eqref{eq:delta_density} using the normalised dispersion relation \eqref{eq:normalised_free_dispersion} as a change of variables. The norm we will consider is the $ L^1$ norm between density functions.

We start by defining the action of $K_l$ on a density function of a random variable $\Lambda$. This is modified from the definition of the Frobenius-Perron operator in ergodic theory \cite{boyarsky1997laws}.
\begin{definition}
    For a measurable function $\tau:[0,\infty)\rightarrow[0,1]$, define $Q_\tau:  L^1[0,\infty) \rightarrow  L^1[0,1]$ such that for any $p\in L^1[0,\infty)$, $Q_{\tau}p$ is the unique (up to almost everywhere equivalence) function in $ L^1[0,1]$ such that
    \begin{equation*}
        \int_B Q_\tau p \,dL = \int_{\tau^{-1}(B)} p\,dL
    \end{equation*}
    for any measurable set $B\subseteq[0,1]$, where $L$ is the Lebesgue measure. Then $\tau^{-1}(B)\subseteq[0,\infty)$ is measurable.
\end{definition}
\begin{remark} \label{rmk:Q_interpretation}
    If $\Lambda$ is a random variable with probability density $p$, then $\tau(\Lambda)$ has probability density $Q_\tau p$. More specifically, if $p_\Lambda$ is the density of $\Lambda$, then $Q_{K_l}p_\Lambda$ is the density of $\kappa_l=K_l(\Lambda)$.
\end{remark}

We want to prove that $Q_{K_l}p$ converge as $l\rightarrow\infty$ for a class of density functions $p$.

\begin{lemma} \label{lm:Qproperties}
    For all measurable $\tau:[0,\infty)\rightarrow[0,1]$, $Q_\tau$ has the following properties:
    \begin{enumerate}
        \item $Q_\tau$ is linear;
        \item The operator norm $\lVert Q_\tau \rVert \leq1$, so $\lVert Q_{\tau}p\rVert_{ L^1[0,1]}\leq\lVert p\rVert_{ L^1[0,\infty)}$ for any $p\in L^1[0,\infty)$.
    \end{enumerate}
\end{lemma}
\begin{proof}
    We may adapt the proofs of similar properties of the Frobenius-Perron operator $P_\tau$ given in Section 4.2 of \cite{boyarsky1997laws}.
\end{proof}

\subsubsection{Functions of bounded variation}

We will split our argument into two cases, depending on whether the measures involved are of bounded variation or not. In the case of bounded variation, the argument is slightly simpler and, crucially, will lead to an explicit estimate of the convergence rate, as indicated in the statements of Theorems \ref{thm:logistic_convergence} \& \ref{thm:generalised_logistic_convergence}. This estimate, we will show, will break down in the presence of unbounded variation. We will exploit the fact that $K_l$ is piecewise linear. We start by considering step functions.
\begin{definition}
    Let $l\in\mathbb{N}$. A function $p:\mathbb{R}\to\mathbb{R}$ is called a rational step function of width $\frac{1}{l}$ on $[0,\infty)$ if
    \begin{equation*}
    p(x) = \sum_{i=0}^N \rho_i\cdot\mathds{1}_{[\frac{i}{l}, \frac{i+1}{l})}(x)
    \end{equation*}
    for some $\rho_i\in\mathbb{R}$ and $N\in\mathbb{N}$.
\end{definition} 
The benefit of considering rational step functions of width $\frac{1}{l}$ is that they are mapped to a constant function by $Q_{K_l}$, which corresponds to the uniform distribution that we want.
\begin{lemma} \label{lm:Rstep_vanish}
    If $p$ is a rational step function of width $\frac{1}{l}$ on $[0,\infty)$ and $p\geq0$, then \begin{equation*}
        Q_{K_{l}} p = \lVert p \rVert_{ L^1[0,\infty)}\cdot\mathds{1}_{[0,1]}.
    \end{equation*}
\end{lemma}
\begin{proof}
    This follows from the fact that on each interval $I = [\frac{i}{l}, \frac{i+1}{l})$ for $i\in\mathbb{N}$, $K_l$ is linear and the image $K_l(I)$ is $[0,1)$ or $(0,1]$.
\end{proof}

For a function of bounded variation $p\in BV[0,\infty)$, we will derive the explicit convergence rate by approximating it with rational step functions.

\begin{lemma} \label{lm:Rstep_approx}
    For every $p\in BV[0,\infty)\cap L^1[0,\infty)$ and $l\in\mathbb{N}$, there exists $p_l$ a rational step function of width $\frac{1}{l}$ on $[0,\infty)$ such that
    \begin{equation*}
        \lVert p_l - p\rVert_{ L^1[0,\infty)} \leq \frac{\ess V(p)}{l},
    \end{equation*}
    where $\ess V(p)=\inf_{\Tilde{p}=p\text{ a.e.}}V(\Tilde{p})$ is the essential variation of $p$ on $[0,\infty)$.
\end{lemma}
\begin{proof}
    Without loss of generality, we write $p$ as the representative of $[p]$ with the smallest variation $V(p) = \ess V([p])$. Let $V(p)_x$ be its variation over $[0,x]$:
    \begin{equation*}
        V(p)_x = \sup_{\mathcal{P}_x}\sum_{i=0}^{N_{\mathcal{P}_x}} |p(a_{i+1})-p(a_i)|,
    \end{equation*}
    where $\mathcal{P}_x$ is a partition of $[0,x]$. Then $V(p)_\infty=V(p)$.
    
    Let $p_1(x) = \frac{1}{2}(V(p)_x+p(x))$ and $p_2(x) = \frac{1}{2}(V(p)_x-p(x))$. Then $p_1$ and $p_2$ are increasing, and $p=p_1-p_2$. Moreover,
    \begin{equation*}
        \lim_{x\rightarrow\infty}p_1(x) = \lim_{x\rightarrow\infty}p_1(x) = \frac{1}{2}V(p).
    \end{equation*}
    This follows from the fact that $\lim_{x\rightarrow\infty}p(x) = 0$. Since $p\in BV[0,\infty)$, $p(x)$ must converge as $x\rightarrow\infty$ (the number of upcrossings of $p(x)$ between any $a<b$ must be finite). Since $p\in L^1[0,\infty)$, the limit must be zero.

    Fix some $l\in\mathbb{N}$. We can approximate the increasing functions $p_1$ and $p_2$ respectively by rational step functions $\Tilde{p_1}$ and $\Tilde{p_2}$ of width $\frac{1}{l}$.
    \begin{align*}
        \Tilde{p_1}(x) &= \sum_{i=0}^\infty p_1(x_i)\cdot\mathds{1}_{[\frac{i}{l},\frac{i+1}{l})},\\
        \Tilde{p_2}(x) &= \sum_{i=0}^\infty p_2(x_i)\cdot\mathds{1}_{[\frac{i}{l},\frac{i+1}{l})},
    \end{align*}
    where each $x_i$ is an arbitrary point in $[\frac{i}{l},\frac{i+1}{l})$.

    Since $p_1$ is increasing, it holds that $|p_1(x)-p_1(x_i)|\leq p_1(\frac{i+1}{l})-p_1(\frac{i}{l})$ for all $i\in\mathbb{N}$ and $x\in[\frac{i}{l},\frac{i+1}{l})$. Therefore,
    \begin{equation*}
        \int_{[\frac{i}{l},\frac{i+1}{l})}|p_1(x)-p_1(x_i)|dx\leq \frac{1}{l}\left(p_1\left(\frac{i+1}{l}\right)-p_1\left(\frac{i}{l}\right)\right)\quad\text{for all $i\in\mathbb{N}$.}
    \end{equation*}
    The sum over $i\in\mathbb{N}$ is telescoping:
    \begin{align*}
        \lVert p_1-\Tilde{p_1}\rVert_{ L^1[0,\infty)} &= \sum_{i=0}^\infty\int_{[\frac{i}{l},\frac{i+1}{l})}|p_1(x)-p_1(x_i)|dx\\
        &\leq\sum_{i=0}^\infty\frac{1}{l}\left(p_1\left(\frac{i+1}{l}\right)-p_1\left(\frac{i}{l}\right)\right)\\
        &=\lim_{i\rightarrow\infty}\frac{1}{l}(p_1(i)-p_1(0))\\
        &= \frac{1}{l}(\frac{1}{2}V(p)-p(0)).
    \end{align*}
    The same result applies for $p_2$:
    \begin{equation*}
        \lVert p_2-\Tilde{p_2}\rVert_{ L^1[0,\infty)}= \frac{1}{l}(\frac{1}{2}V(p)+p(0)).
    \end{equation*}

    Let $p_l = \Tilde{p_1} - \Tilde{p_2}$. Then $p_l$ is also a rational step function of width $\frac{1}{l}$ and we have the bound
    \begin{equation*}
        \lVert p-p_l\rVert_{ L^1[0,\infty)} \leq \lVert p_1-\Tilde{p_1}\rVert_{ L^1[0,\infty)} + \lVert p_2-\Tilde{p_2}\rVert_{ L^1[0,\infty)} \leq \frac{V(p)}{l}.
    \end{equation*}
\end{proof}

Combining the previous results gives an explicit estimate on the convergence rate of $Q_{K_l}p$ to uniformity as $l\rightarrow\infty$:
\begin{theorem} \label{thm:bv_convergence_rate}
    For a density function $p\in BV[0,\infty)\cap L^1[0,\infty)$ such that $p\geq 0$ and $\lVert p \rVert_{ L^1[0,\infty)}=1$,
    \begin{equation*}
        \lVert Q_{K_l}p-\mathds{1}\rVert_{ L^1[0,1]}\leq\frac{2\cdot\ess V(p)}{l}\quad\text{for all $l\in\mathbb{N}$}.
    \end{equation*}
\end{theorem}
\begin{proof}
    For brevity we will write $\lVert\cdot\rVert_1=\lVert\cdot\rVert_{ L^1[0,1]}$ and $\lVert\cdot\rVert_2=\lVert\cdot\rVert_{ L^1[0,\infty)}$. Fix $l\in\mathbb{N}$. By Lemma~\ref{lm:Rstep_approx}, there exists $p_l$ a rational step function of width $\frac{1}{l}$ on $[0,\infty)$ such that $\lVert p-p_l\rVert_2\leq\frac{V(p)}{l}$. Without loss of generality, assume $p_l\geq0$ (if not, take $\Tilde{p_l} = \max(p_l,0)$).
    
    Thanks to the linearity of $Q_{K_l}$, we have the bound
    \begin{align*}
        \lVert Q_{K_l}p-\mathds{1}\rVert_1 \leq \lVert Q_{K_l}(p-p_l)\rVert_1 + \lVert Q_{K_l}p_l-\lVert p_l\rVert_2\cdot\mathds{1}\rVert_1+|\lVert p_l\rVert_2-1|.
    \end{align*}
    Since $p_l$ is a rational step function of width $\frac{1}{l}$, Lemma~\ref{lm:Rstep_vanish} implies that $Q_{K_l}p_l=\lVert p_l\rVert_2\cdot\mathds{1}$, so
    \begin{equation*}
        \lVert Q_{K_l}p-\mathds{1}\rVert_1 \leq \lVert Q_{K_l}(p-p_l)\rVert_1 + |\lVert p_l\rVert_2-1|.
    \end{equation*}
    By Lemma \ref{lm:Qproperties},
    \begin{equation*}
        \lVert Q_{K_l}p-\mathds{1}\rVert_1 \leq \lVert p-p_l\rVert_2 + |\lVert p_l\rVert_2-1|.
    \end{equation*}
    Since $\lVert p\rVert_2=1$ and by the reverse triangular inequality, $|\lVert p_l\rVert_2-1| \leq \lVert p-p_l\rVert_2$. Therefore,
    \begin{equation*}
        \lVert Q_{K_l}p-\mathds{1}\rVert_1 \leq 2\lVert p-p_l\rVert_2 \leq \frac{2\cdot\ess V(p)}{l}.
    \end{equation*}
\end{proof}
\begin{remark}
    In our construction of the generalised logistic maps, we consider an exponentially growing sequence of cell lengths $(l_n=m^n:{n\in\mathbb{N}})$ for a fixed $m\in\mathbb{N}$. This results in a convergence rate that is also exponential:
    \begin{equation*}
        \lVert Q_{K_{m^n}}p-\mathds{1}\rVert_{ L^1[0,1]}\leq\frac{2\cdot\ess V(p)}{m^n}\quad\text{for all $n\in\mathbb{N}$}.
    \end{equation*}
\end{remark}

\subsubsection{Functions of unbounded variation}

For a general $p\in L^1[0,\infty)$, the approximation by rational step functions in Lemma~\ref{lm:Rstep_approx} breaks down. Instead, we will use general step functions to handle the more general case of functions of potentially unbounded variation.
\begin{lemma} \label{lm:genstepapprox}
    Let $I\subseteq\mathbb{R}$ be an interval. For every bounded non-negative function $p\in L^1(I)$ such that $0\leq p \leq M$ for some $M>0$, there exists a sequence of step functions $p_j$ on $I$ such that $0\leq p_j \leq M$ and $p_j \rightarrow p$ as $j\rightarrow\infty$ almost everywhere.
\end{lemma}
\begin{proof}
    By \cite[Theorem 4.3]{stein2005real}, there exists a sequence of step functions $\Tilde{p_j}$ such that $\Tilde{p_j}\rightarrow p$ as $k\rightarrow\infty$ almost everywhere. To get the desired $p_j$, either observe that $\Tilde{p_j}$ is already bounded in $[0,M]$ by construction, or simply take $p_j = \min(M,\max(0, \Tilde{p_j}))$.
\end{proof}

\begin{theorem} \label{thm:general_convergence}
    For any density function $p\in L^1[0,\infty)$ such that $p\geq0$ and $\lVert p\rVert_{ L^1[0,\infty)}=1$,
    \begin{equation*}
        Q_{K_{l}} p \rightarrow \mathds{1}\text{ as $l\rightarrow\infty$}\quad\text{in $ L^1[0,1]$.}
    \end{equation*}
\end{theorem}
\begin{proof}
    Again, for simplicity we will write $\lVert\cdot\rVert_1=\lVert\cdot\rVert_{ L^1[0,1]}$ and $\lVert\cdot\rVert_2=\lVert\cdot\rVert_{ L^1[0,\infty)}$. We first consider bounded $p$. Say $p$ is bounded such that $p\leq M$. By Lemma~\ref{lm:genstepapprox}, there exist step functions $p_j\in L^1[0,\infty)$ such that $0\leq p_j\leq M$ and $p_j\rightarrow p$ as $j\rightarrow\infty$ almost everywhere. By the Dominated Convergence Theorem, $p_j\rightarrow p$ in $ L^1[0,\infty)$. Then for any $\epsilon>0$, there exists $J\in\mathbb{N}$ such that $|\lVert p_J\rVert_2-1| \leq \lVert p_J-p\rVert_2 \leq \epsilon$.

    Thanks to the linearity of $Q_{K_l}$ and the triangle inequality,
    \begin{equation*}
        \lVert Q_{K_l} p -\mathds{1}\rVert_1 \leq \rVert Q_{K_l} (p - p_J)\rVert_1 +\lVert Q_{K_l} p_J - \lVert p_J\rVert_2\mathds{1}\rVert_1 + \lVert \lVert p_J\rVert_2\mathds{1}-\mathds{1}\rVert_1.
    \end{equation*}
    By Lemma~\ref{lm:Qproperties},
    \begin{align*}
        \lVert Q_{K_l} p -\mathds{1}\rVert_1 &\leq\lVert p-p_J\rVert_2 + \lVert Q_{K_l} p_J - \lVert p_J\rVert_2\mathds{1}\rVert_1+|\lVert p_J\rVert_2-1|\\
        &\leq 2\epsilon+\lVert Q_{K_l} p_J - \lVert p_J\rVert_2\mathds{1}\rVert_1.
    \end{align*}
    Any step function $p_J$ clearly has bounded variation, so by Theorem~\ref{thm:bv_convergence_rate}, $\lVert Q_{K_l} p_J-\lVert p_J\rVert_2\mathds{1}\rVert_1 \leq \frac{c_J}{l}$ for all $l\in\mathbb{N}$ where $c_J=2\cdot\ess V(p_J)$. Then,
    \begin{equation*}
        \lVert Q_{K_l} p -\mathds{1}\rVert_1 \leq2\epsilon+\frac{c_J}{l}\quad\text{for all $l\in\mathbb{N}$}.
    \end{equation*}
    There exists $l_\epsilon\in\mathbb{N}$ such that for all $l\geq l_\epsilon$ it holds that $\lVert Q_{K_l}p -\mathds{1}\rVert_1\leq3\epsilon$. Therefore, for bounded $p$,
    \begin{equation*}
        Q_{K_l}p\rightarrow\mathds{1}\text{ as $l\rightarrow\infty$}\quad\text{in $ L^1[0,1]$.}
    \end{equation*}

    For a general density function $p\in L^1[0,\infty)$, let $p^M = \min(p,M)$. Then $0\leq p^M \nearrow p$, so by the Monotone Convergence Theorem, $p^M \rightarrow p$ in $ L^1[0,\infty)$ as $M\rightarrow\infty$.

    For all $\epsilon>0$, there exists $M_\epsilon>0$ such that $|\lVert p^{M_\epsilon}\rVert_2-1| \leq \lVert p^{M_\epsilon}-p\rVert_2 \leq \epsilon$. $p^{M_\epsilon}$ is bounded, so we can apply the results above. There exists $l_\epsilon\in\mathbb{N}$ such that for all $l\geq l_\epsilon$, 
    \begin{equation*}
        \lVert Q_{K_l} p^{M_\epsilon} -\lVert p^{M_\epsilon}\rVert_2\cdot\mathds{1}\rVert_1 \leq \epsilon.
    \end{equation*}
    Again, by the triangle inequality and Lemma~\ref{lm:Qproperties},
    \begin{align*}
        \lVert Q_{K_l} p -\mathds{1}\rVert_1 &\leq \lVert  Q_{K_l} (p -p^{M_\epsilon})\rVert_1 + \lVert  Q_{K_l} p^{M_\epsilon} -\lVert p^{M_\epsilon}\rVert_2\mathds{1}\rVert_1 + |\lVert p^{M_\epsilon}\rVert_2 - \lVert p \rVert_2|\\
        &\leq\lVert p-p^{M_\epsilon}\rVert_2 + \lVert  Q_{K_l} p^{M_\epsilon} -\lVert p^{M_\epsilon}\rVert_2\mathds{1}\rVert_1 + |\lVert p^{M_\epsilon}\rVert_2 - \lVert p \rVert_2|\\
        &\leq 3\epsilon,
    \end{align*}
    for all $l\geq l_\epsilon$. Therefore, $Q_{K_l} p\rightarrow\mathds{1}$ as $l\rightarrow\infty$ in $ L^1[0,1]$.
\end{proof}

\begin{remark}
    We present an example of $p\in L^1[0,\infty)$ with unbounded variation such that the convergence rate of $Q_{K_l}p$ in Theorem~\ref{thm:bv_convergence_rate} breaks down. Consider
    \begin{equation*}
        p(x) = \sum_{i=0}^\infty 2^{\frac{i}{2}} \mathds{1}_{(\frac{1}{2^{i+1}},\frac{1}{2^i}]}\in L^1[0,\infty),
    \end{equation*}
    which has norm
    \begin{equation*}
        \lVert p\rVert_2 = \sum_{i=0}^\infty 2^\frac{i}{2}\cdot\frac{1}{2^{i+1}} = \frac{1}{2}\sum_{i=0}^{\infty} 2^{-\frac{i}{2}} = \frac{\sqrt{2}+1}{2\sqrt{2}}.
    \end{equation*}
    Moreover, $p$ has unbounded variation since $p$ is unbounded near 0. By linearity of $Q_{K_l}$,
    \begin{align*}
        Q_{K_{2^n}}p &= Q_{K_{2^n}}\left(\sum_{i=0}^{i=n-1}2^{\frac{i}{2}}\mathds{1}_{(\frac{1}{2^{i+1}},\frac{1}{2^i}]}\right) + Q_{K_{2^n}}\left(\sum_{i=n}^\infty 2^{\frac{i}{2}} \mathds{1}_{(\frac{1}{2^{i+1}},\frac{1}{2^i}]}\right)\\
        &= \sum_{i=0}^{i=n-1} 2^{\frac{i}{2}}\cdot\frac{1}{2^{i+1}}\mathds{1}_{(0,1]} + \sum_{i=n}^\infty 2^{\frac{i}{2}-n}\mathds{1}_{(\frac{1}{2^{i-n+1}},\frac{1}{2^{i-n}}]} \\
        &= \frac{1}{2}\sum_{i=0}^{i=n-1} 2^{-\frac{i}{2}}\mathds{1}_{(0,1]} + \sum_{i=n}^\infty 2^{\frac{i}{2}-n}\mathds{1}_{(\frac{1}{2^{i-n+1}},\frac{1}{2^{i-n}}]}.
    \end{align*}
    Therefore,
    \begin{align*}
        \left\lVert Q_{K_{2^n}}p-\lVert p\rVert_2\mathds{1}_{[0,1]}\right\rVert_1 = &\left\lVert- \frac{1}{2}\sum_{i=n}^{\infty} 2^{-\frac{i}{2}}\mathds{1}_{(0,1]} + \sum_{i=n}^\infty 2^{\frac{i}{2}-n}\mathds{1}_{(\frac{1}{2^{i-n+1}},\frac{1}{2^{i-n}}]} \right\rVert_1 \\
        = &\left\lVert \sum_{i=n}^\infty \left(2^{\frac{i}{2}-n} - \frac{1}{2}\sum_{j=n}^{\infty} 2^{-\frac{j}{2}}\right)\mathds{1}_{(\frac{1}{2^{i-n+1}},\frac{1}{2^{i-n}}]} \right\rVert_1 \\
        = &\sum_{i=n}^\infty \left|2^{\frac{i}{2}-n}- \frac{1}{2}\sum_{j=n}^{\infty} 2^{-\frac{j}{2}}\right|\cdot\frac{1}{2^{i-n+1}}  \\
        = &\frac{2+\sqrt{2}}{4}\cdot2^{-\frac{n}{2}}.
    \end{align*}
    Thus $\lVert Q_{K_{2^n}}\frac{p}{\lVert p\rVert_2}-\mathds{1}\rVert_1\sim O(2^{-\frac{n}{2}})$ instead of $O(2^{-n})$. 
\end{remark}

\subsubsection{Proofs of Theorems \ref{thm:logistic_convergence} \& \ref{thm:generalised_logistic_convergence} }

We are now ready to translate these results back to the (generalised) logistic maps. We handle the proof of Theorem~\ref{thm:generalised_logistic_convergence} first, of which Theorem~\ref{thm:logistic_convergence} is a special case (after a suitable change of variables).

\begin{proof}[Proof of Theorem \ref{thm:generalised_logistic_convergence}]
    In fact, the $ L^1(I)$ norm on an interval $I$ coincides with the total variation distance between two measures on $I$:
    \begin{equation*}
        \delta(P,Q) \coloneq \sup_{B\in\mathcal{B}} |P(B)-Q(B)|,
    \end{equation*}
    where $\mathcal{B}$ is the Borel $\sigma$-algebra of $I$. It is clear by definition that the total variation distance is independent of any bijective and measurable change of variables. Moreover, it coincides with the $L^1$ distance when measures $P$ and $Q$ are absolutely continuous with density functions $p$ and $q$ \cite{tsybakov2009introduction}.
    \begin{equation*}
        \delta(P,Q) = \frac{1}{2}\int_I |p-q|dL=\frac{1}{2}\lVert p-q\rVert_{ L^1(I)}.
    \end{equation*}
    Thus, the $L^1$ distance is independent of any bijective and measurable change of variables.
    
    Recall the normalised dispersion relation \eqref{eq:normalised_free_dispersion}. This defines a bijective and measurable change of variables between $\Delta\in[-2,2]$ and $\kappa\in[-1,1]$, given by $\Delta=2\cos(\pi\kappa)$ and $\kappa = \cos^{-1}(\Delta/2)/\pi$. Moreover, the invariant density $D$ in \eqref{eq:delta_density} corresponds to a uniform $\kappa$ after this change of variables:
    \begin{equation*}
        \mathds{1}(\kappa)\cdot\left|\frac{\de\kappa}{\de\Delta}\right| = \left|\frac{\de}{\de\Delta}\left(\frac{1}{\pi}\cos^{-1}\left(\frac{\Delta}{2}\right)\right)\right| = \frac{1}{\pi}\frac{1}{\sqrt{4-\Delta^2}} = D(\Delta).
    \end{equation*}
    Suppose $p_{\kappa}$ is any probability density of $\kappa$ on $[0,1]$ and $p_{\Delta}$ is the corresponding density of $\Delta$ on $[-2,2]$, then
    \begin{equation*}
        \lVert p_{\kappa}-\mathds{1}\rVert_{ L^1[0,1]} = \lVert p_{\Delta}-D\rVert_{ L^1[-2,2]}.
    \end{equation*}
    Thus, our results about convergence of $\kappa_l$ immediately apply to the convergence of $\Delta_l$.

    Recall that $p_{\kappa_l}=Q_{K_l}p_\Lambda$ by Remark \ref{rmk:Q_interpretation}, where $p_\Lambda$ is obtained from $p_{\Delta_1}$ by the relation \eqref{eq:Lambda_Delta_relation}. Therefore, by Theorem \ref{thm:general_convergence},
    \begin{equation*}
        \lVert p_{\Delta_{l}}-D\rVert_{ L^1[-2,2]}=\lVert p_{\kappa_{l}}-\mathds{1}\rVert_{ L^1[0,1]} = \lVert Q_{K_l}p_\Lambda-\mathds{1}\rVert_{ L^1[0,1]} \rightarrow 0\quad\text{as $l\rightarrow\infty$.}
    \end{equation*}

    If further $p_{\Delta_1}$ has bounded variation, then $p_\Lambda\in BV[0,\infty)\cap L^1[0,\infty)$ and by Theorem \ref{thm:bv_convergence_rate},
    \begin{equation*}
        \lVert p_{\Delta_{l}}-D\rVert_{ L^1[-2,2]}= \lVert Q_{K_l}p_\Lambda-\mathds{1}\rVert_{ L^1[0,1]} \leq\frac{2\cdot\ess V(p_\Lambda)}{l}\quad\text{for all $l\in\mathbb{N}$.}
    \end{equation*}

    Going back to the recursion relation, we replace $l$ with the sequence $(l_n=m^n:n\in\mathbb{N})$. This gives us the desired convergence results with constant $\Tilde{c}=2\cdot\ess V(p_\Lambda)$.
\end{proof}

\begin{remark} \label{rmk:constant_initial}
    We can express the constant $\Tilde{c}$ in terms of the initial density $p_{\Delta_1}$. One way to do this is to pick $\Lambda$ such that $\Lambda=0$ outside of the first branch $[0,1]$, then $\Lambda=\cos^{-1}(\Delta_1/{2})/\pi$ and
    \begin{equation*}
        p_\Lambda(x)= p_{\Delta_1}(2\cos(\pi x))\cdot\left|\frac{\de\Delta_1}{\de\Lambda}(x)\right| = p_{\Delta_1}(2\cos(\pi x))\cdot \left|2\pi\sin(\pi x)\right|\quad\text{for $x\in[0,1].$}
    \end{equation*}
    A generous inequality is
    \begin{align*}
        V(p_\Lambda) &= \sup_{\mathcal{P}_x}\sum_{i=0}^{N_{\mathcal{P}_x}} |p_\Lambda(a_{i+1})-p_\Lambda(a_i)|\\
        &\leq \sup_{\mathcal{P}_x}\sum_{i=0}^{N_{\mathcal{P}_x}}2\pi|\sin(\pi x_{i+1})||p_{\Delta_1}(\Tilde{x}_{i+1})-p_{\Delta_1}(\Tilde{x}_{i})|+ 2\pi|\sin(\pi x_{i+1})-\sin(\pi x_i)||p_{\Delta_1}(\Tilde{x}_i)|\\
        &\leq2\pi\cdot V(p_{\Delta_1}) + 4\pi\cdot\sup\,p_{\Delta_1}
    \end{align*}
    where $\Tilde{x}=2\cos(\pi x)$ for simplicity. Note that $\sup p_{\Delta_1}$ is finite when it has bounded variation. Thus, 
    \begin{align*}
        \Tilde{c} = 2\cdot\ess V(p_\Lambda)\leq4\pi\cdot\ess V(p_{\Delta_1})+8\pi\cdot\ess\sup(p_{\Delta_1}).
    \end{align*}
\end{remark}

\begin{remark} \label{rmk:exponential_conv}
    In fact, the exponential convergence rate $O(m^{-n})$ works for a wider class than $p_{\Delta_1}\in BV[-2,2]$. Some functions that blow up at the edges with appropriate speed will have bounded variation after the change of variables. For example, even the invariant density $p_{\Delta_1}=D$ is unbounded at $-2$ and $2$, but the corresponding $p_\Lambda$ is piecewise linear and thus has bounded variation. Therefore, it is more informative to consider the variation of $p_\Lambda$ directly.
\end{remark}

\begin{proof}[Proof of Theorem \ref{thm:logistic_convergence}]
    This is a special case of Theorem \ref{thm:generalised_logistic_convergence}. Suppose $X_n = F^n(X_0)$ is a sequence that follows the logistic map with $r=4$. Recall that $X_n$ is related to $\Delta_{2^n}$ by $\Delta_{2^n} = 2-4X_n$ \eqref{eq:logistic_mandelbrot_conjugacy}. Again, since the total variation norm is independent of a bijective change of variables, 
    \begin{equation*}
        \lVert p_{X_n}-q\rVert_{L^1[0,1]} = \lVert p_{\Delta_{2^n}} - D\rVert_{L^1[-2,2]},
    \end{equation*}
    where $q$ is the invariant density of the logistic map \eqref{eq:logistic_invariant_measure}. Then, this reduces to the $m=2$ case in Theorem~\ref{thm:generalised_logistic_convergence}.
\end{proof}

The constant $c$ in Theorem \ref{thm:logistic_convergence} will be $2 \ess V(p_\Lambda)$, where $p_\Lambda$ is again the density of $\Lambda$ given by \eqref{eq:Lambda_Delta_relation}, and Remarks~\ref{rmk:constant_initial} and~\ref{rmk:exponential_conv} apply to this case as well.

\subsection{Explicit formulas} \label{sec:explicit_formulas}
In this section we explore another tangential consequence of our method: knowledge of exact solutions for the Hill operators can be used to derive explicit formulas for orbits of the recursion relations. Recall that logistic map with $r=4$ has the explicit formula \eqref{eq:explicit_formula_sine}. We will present a new proof of this famous result using Hill operators. Our method also provides a framework for generating other formulas and we present an example that involves Mathieu functions.

\subsubsection{The well-known formula}
Take the free Hill operator again, or any homogeneous $V\equiv a\in\mathbb{R}$. As solved in the last section, the dispersion relation is
\begin{equation} \label{eq:homogeneous_discriminant}
    \Delta_l(\lambda) =2\cos(\omega l),
\end{equation}
where $\omega=\sqrt{\lambda-a}$ for $\lambda\geq a$.

Recall that $(\Delta_{2^n}:n\in\mathbb{N})$ is related to the logistic map by the change of variables $x_n = \frac{1}{4}(2-\Delta_{2^n})$. Substituting this into \eqref{eq:homogeneous_discriminant} gives a formula for $x_n$ in terms of $\omega$:
\begin{align} \label{eq:1}
    \begin{split}
        x_n &= \frac{1}{4}(2-2\cos{2^n\omega})\\
        &=\frac{1}{4}[2-2(2\cos^2{(2^{n-1}\omega)}-1)]\\
        &=1-\cos^2{(2^{n-1}\omega)}\\
        &=\sin^2{(2^{n-1}\omega)}\quad\text{for all $n\geq1$.} 
    \end{split}
\end{align}
All that remains is to find $\omega$ in terms of the initial value $x_0$. Note that $x_0 = \frac{1}{4}(2-\Delta_1) = \frac{1}{4}(2-2\cos{\omega}) = \frac{1}{2}(1-\cos{\omega})$, so
\begin{equation*}
    \omega = 2\cdot\sin^{-1}(\sqrt{x_0}).
\end{equation*}
Substituting this expression for $\omega$ in \eqref{eq:1} gives the desired explicit formula \eqref{eq:explicit_formula_sine}.

\subsubsection{A new formula with Mathieu functions} \label{sec:Mathieu}

A benefit of this framework is that new formulas can be generated by choosing other periodic potential functions. In general, we are not able to characterise eigensolutions of the Hill operator explicitly. However, there are a few well studied exceptions. For example, if we instead choose $V(x) = \cos(2\pi x)$ which is 1-periodic, then the eigenvalue problem becomes
\begin{equation*}
    \frac{d^2y}{dx^2}+\cos(2\pi x)\cdot y = \lambda y.
\end{equation*}
After a change of variables $v =\pi x$, this becomes the Mathieu differential equation:
\begin{equation} \label{eq:Mathieu}
    \frac{d^2y}{dv^2} + \left(a-2q\cdot\cos{2v}\right)y=0\quad\text{where $a=-\frac{\lambda}{\pi^2}$ and $q=-\frac{1}{2\pi^2}$.}
\end{equation}

Solutions to Mathieu's equation \eqref{eq:Mathieu} are well studied  \cite{wolf2009mathieu}. Let $C(a,q,v)$ be the even solutions and $S(a,q,v)$ be the odd solutions. Then,
\begin{equation*}
    \varphi_\lambda(x) = \frac{C(a,q,v)}{C(a,q,0)} = \frac{C(a,q,\pi x)}{C(a,q,0)},\quad
    \psi_\lambda(x) = \frac{S(a,q,v)}{\pi S_v(a,q,0)} = \frac{S(a,q,\pi x)}{\pi S_v(a,q,0)}.
\end{equation*}
are the basic solutions satisfying the initial conditions \eqref{eq:basis_solutions_bc}. The monodromy matrix is
\begin{equation*}
    M_l(\lambda)\coloneq\begin{bmatrix}
\varphi_\lambda(l) & \psi_\lambda(l) \\
\varphi_\lambda'(l) & \psi_\lambda'(l) 
\end{bmatrix}
= \begin{bmatrix}
\frac{C(a,q,\pi l)}{C(a,q,0)} & \frac{S(a,q,\pi l)}{\pi S_v(a,q,0)} \\
\frac{\pi C_v(a,q,\pi l)}{C(a,q,0)} & \frac{S_v(a,q,\pi l)}{S_v(a,q,0)}
\end{bmatrix},
\end{equation*}
and the discriminant is
\begin{equation*}
    \Delta_{2^n}(\lambda) = \tr(M(2^n)) = \frac{C(a,q,2^n\pi)}{C(a,q,0)} + \frac{S_v(a,q,2^n\pi)}{S_v(a,q,0)}.
\end{equation*}

When $n=0$, substitute in the values of $a$ and $q$ from \eqref{eq:Mathieu}. This gives us a relation between $x_0$ and $\lambda$:
\begin{equation} \label{eq:mathieu_initial_relation}
    x_0 = \frac{1}{4}(2-\Delta_1) = \frac{1}{4}\left(2-\frac{C(-\frac{\lambda}{\pi^2},-\frac{1}{2\pi^2},\pi)}{C(-\frac{\lambda}{\pi^2},-\frac{1}{2\pi^2},0)} - \frac{S_v(-\frac{\lambda}{\pi^2},-\frac{1}{2\pi^2},\pi)}{S_v(-\frac{\lambda}{\pi^2},-\frac{1}{2\pi^2},0)}\right) \eqcolon f(\lambda).
\end{equation}
There exists $\lambda_0<0$ such that the function $f:[-\lambda_0,0]\rightarrow[0,1]$ is bijective (see Figure~\ref{fig:mathieu_initial_relation}). Therefore, the inverse $f^{-1}:[0,1]\rightarrow[-\lambda_0,0]$ is well-defined, and we have a new formula:
\begin{equation*}
    x_n = \frac{1}{4}(2-\Delta_{2^n}) = \frac{1}{4}\left(2-\frac{C(a,q,2^n\pi)}{C(a,q,0)} - \frac{S_v(a,q,2^n\pi)}{S_v(a,q,0)}\right),
\end{equation*}
where
\begin{equation*}
    a=-\frac{f^{-1}(x_0)}{\pi^2}, \quad q=-\frac{1}{2\pi^2}.
\end{equation*}

\begin{figure}
    \centering
    \includegraphics[width=0.4\linewidth]{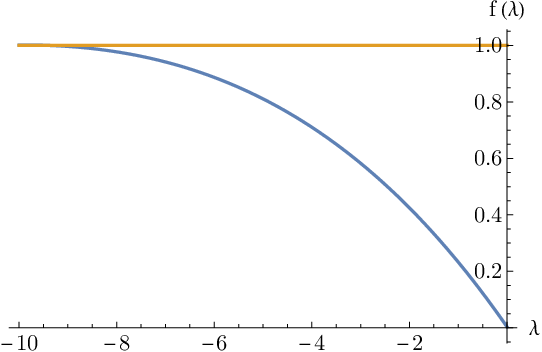}
    \caption{The graph of the relation $x_0=f(\lambda)$ in \eqref{eq:mathieu_initial_relation}. There exists $\lambda_0\in(-10,-9)$ such that $f(\lambda_0)=1$ and $f$ is a bijection from $[\lambda_0,0]$ to $[0,1]$.}
    \label{fig:mathieu_initial_relation}
\end{figure}

\section{Conjugacy to other maps} \label{sec:conjugacy}

To understand the wider implications of this work, it is useful to consider the other maps that can be obtained through suitable changes of variables. The increasingly steep linear peaks of the normalised spectral bands in Figure~\ref{fig:folding}(d)-(f) resemble the multiple tent map (as defined in \eqref{eq:K_l} and depicted in Figure~\ref{fig:folding}(d)-(f)). In fact, the process of normalising and linearising the spectral bands aligns with the idea of conjugacy in ergodic theory.

For general $m\in\mathbb{N}$, we define the multiple-tent map with $\frac{m}{2}$ ``tents'' on $[0,1]$ as in Figure \ref{fig:gm}:
\begin{equation} \label{eq:gm}
    g_m(x)\coloneq
    \begin{cases}
        mx-2j & \text{if $x\in[\frac{2j}{m},\frac{2j+1}{m}]$ for $j=0,1,...,\lfloor\frac{m-1}{2}\rfloor$}, \\
        -mx+2j & \text{if $x\in[\frac{2j-1}{m},\frac{2j}{m}]$ for $j=1,2,...,\lfloor\frac{m}{2}\rfloor$}.
    \end{cases}
\end{equation}
This is exactly $K_m$ \eqref{eq:K_l} restricted to $[0,1]$:
\begin{equation*}
    K_m(x)=g_m(\lfloor x\rfloor)\quad\text{for all $x\in\mathbb{R}$.}
\end{equation*}

\begin{figure}
    \centering
    \includegraphics[width=0.45\linewidth]{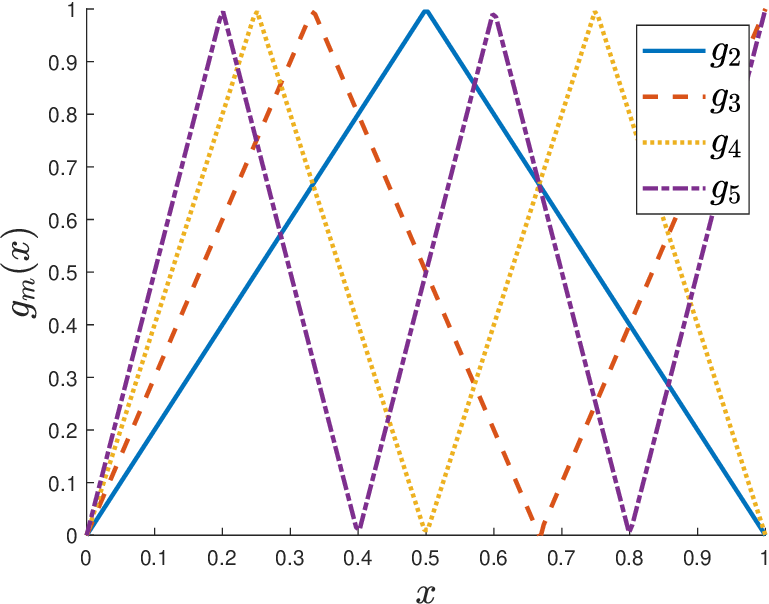}
    \caption{The multiple-tent map $g_m$ for $m=2,3,4,5$.}
    \label{fig:gm}
\end{figure}

\begin{theorem}
    The generalised logistic maps $f_m$ are conjugate to the multiple-tent maps $g_m$ through the bijection
    \begin{equation*}
    C(x) \coloneq 2\cos(\pi x)
    \end{equation*}
    for all $m\in\mathbb{N}$.
\end{theorem}

\begin{proof}
    Substituting the dispersion relation \eqref{eq:dispersion} to the definition of $f_m$ \eqref{eq:generalised_logistic_def} gives
    \begin{equation} \label{eq:conjugacy}
        2\cos(mk) = \Delta_m = f_m(\Delta_1) = f_m(2\cos k)\quad\text{for all $k\in\mathbb{R}$.}
    \end{equation}
    Then, substituting the cosine terms by $C$ gives
    \begin{equation*}
        C\circ g_m(x) = f_m\circ C(x)\quad\text{for all $x\in[0,1]$.}
    \end{equation*}
\end{proof}

It is clear that the invariant measure of $g_m$ is uniform $\mathcal{U}[0,1]$. Thanks to the similarity between $g_m$ and $K_m$, our convergence results in Section \ref{sec:convergence} all apply for the multiple-tent map, after small modifications to the domain. Moreover, the generalised logistic maps $f_m$ are also conjugate to the Chebyshev polynomials of the first kind, which are defined as multiple-angle formulas for cosine:
\begin{equation*}
    T_m(\cos k) = \cos{mk}\quad\text{for $m\in\mathbb{N}$}.
\end{equation*}

\begin{theorem}
    $f_m$ is conjugate to $T_m$ by $E(x) = 2x$ for all $m\in\mathbb{N}$, i.e. $f_m$ is a rescaled version of $T_m$.
\end{theorem}
\begin{proof}
    This is also clear from \eqref{eq:conjugacy} which resembles the multiple-angle formula.
\end{proof}

One could define general Chebyshev maps that are no longer polynomials by taking any $m>1$ instead of $m\in\mathbb{N}$. The invariant measures of such Chebyshev maps are studied in \cite{boyarsky2001invariant}.

\section{Lyapunov exponents} \label{sec:lyapunov}

When studying chaotic dynamics, the Lyapunov exponent is an important quantity that 
measures the exponential divergence of nearby trajectories. In this section, we will explore the Lyapunov exponents of the generalised logistic maps. We follow the definitions given in \cite{hilborn2000quantifying}. In a deterministic dynamical system with discrete time, $x_{n+1}=f(x_n)$, the local Lyapunov exponent is given by
\begin{equation*}
    \lambda(x) = \log |f'(x)|,
\end{equation*}
when $f'$ is defined. Then, the more well-known average Lyapunov exponent is given by
\begin{equation} \label{eq:average_lyapunov}
    \lambda = \int_M\lambda(x)\cdot p(x)dx,
\end{equation}
where $M$ is the domain of $f$ and $p$ is an invariant density of $f$.

For the generalised logistic map of degree $m$, the average Lyapunov exponent is
\begin{equation*}
    \begin{split}
        \lambda_m &= \int_{-2}^2 \log|f_m'(x)|\cdot D(x)dx \\
        &= \int_{-2}^2 \log|m\prod_{i=0}^{m-1}(x-a_{m,i})|\cdot D(x)dx \\
        &= \log m + \sum_{i=0}^{m-1}I(a_{m,i}),\\
    \end{split}
\end{equation*}
where $D$ is our invariant density \eqref{eq:delta_density} and
\begin{equation}
    I(a) = \int_{-2}^2 \log|x-a|\cdot D(x)dx = \frac{1}{\pi}\int_{-2}^2\frac{\log|x-a|}{\sqrt{4-x^2}}dx.
\end{equation}
After the change of variables $x=2\sin y$, this can also be written as
\begin{equation} \label{eq:Ia}
        I(a) = \frac{1}{\pi}\int_{-\pi /2}^{\pi/2}\frac{\log|2\sin y-a|}{2\cos y}\cdot2\cos y \,\de y 
        = \frac{1}{\pi}\int_{-\pi/2}^{\pi/2} \log|2\sin y-a|\,\de y.
\end{equation}
We observe that $I(a)$ vanishes for $a\in(-2,2)$, as shown in Figure \ref{fig:vanishing_integrals}. Verifying that $I(0)=0$ is relatively straightforward, based on standard integrals. Once we have computed the Lyapunov exponent, below, we will also be able to show that $I(\pm1)=0$, see Appendix~\ref{app:integrals} for details.

\begin{figure}
    \centering
    \includegraphics[width=0.5\linewidth]{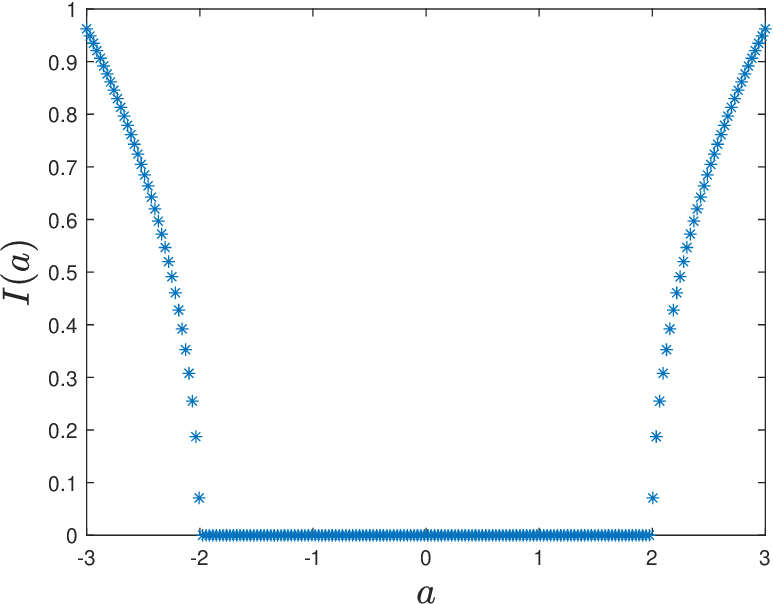}
    \caption{$I(a)$, as defined in \eqref{eq:Ia}, calculated numerically using MATLAB's built-in adaptive quadrature method. The function $I$ is even and vanishes for $-2\leq a\leq2$.}
    \label{fig:vanishing_integrals}
\end{figure}

We can use the conjugacy discussed in Section~\ref{sec:conjugacy} to estimate the Lyapunov exponents. We will assume that the Lyapunov exponent is invariant under the conjugacy $C$. This is a reasonable assumption since \cite{eichhorn2001lyapunov} proved that in continuous dynamical systems, the Lyapunov exponent is invariant under an invertible transformation. \cite{kuznetsov2016invariance} proved this for diffeomorphisms and claimed that similar results are true for discrete systems. In our case the conjugacy $C$ is bijective and smooth on $(0,1)$. Therefore, we assume that the Lyapunov exponent for $f_m$ and that for $g_m$ are the same.

\begin{theorem} \label{thm:lyapunov}
    The average Lyapunov exponent of both $f_m$ and $g_m$ is
    \begin{equation*}
        \lambda_m = \log m.
    \end{equation*}
\end{theorem}
\begin{proof}
    Since $g_m$ is piecewise linear with slope $\pm m$, its local Lyapunov exponent is $\log m$ on the entire domain $[0,1]$, except the $m-1$ jump discontinuity points. Thus, the average in \eqref{eq:average_lyapunov} is also $\log m$.
\end{proof}

The Lyapunov exponent is often used as a measure of chaos. For example, relations between the Lyapunov exponent and concepts like mixing and decorrelation are proven in \cite{collet2004liapunov,butkovskii2008mixing}. In our case, it is easy to see that the convergence rates derived in Theorems~\ref{thm:logistic_convergence} and~\ref{thm:generalised_logistic_convergence} agree with the average Lyapunov exponent from Theorem~\ref{thm:lyapunov}, since $m^{-n} = \exp(-n\log(m)) = \exp(-n\lambda_m)$.

We can verify our predictions of the convergence rate numerically. As a demonstrative example, we computed the distributions of the first few $X_n=f_m^n(X_0)$ when $X_0\sim \Gamma(1,1)-2$ (\emph{i.e.} $X+2$ satisfies the gamma distribution $\Gamma(1,1)$). In Figure~\ref{fig:convergence_rate} the Wasserstein distance is used to compare the distribution of $X_n$ with the invariant measure. For each different degree $m$, we observe exponential convergence before the error stabilises. The rate of the exponential convergence varies with the degree $m$. This is shown in Figure~\ref{fig:slopes}, from which we can see that it is proportional to $\log m$, as predicted by both our convergence result Theorem~\ref{thm:generalised_logistic_convergence} and the average Lyapunov exponent from Theorem~\ref{thm:lyapunov}.

\begin{figure}
    \centering
    \begin{subfigure}[b]{0.45\textwidth}
        \centering
        \includegraphics[width=\linewidth]{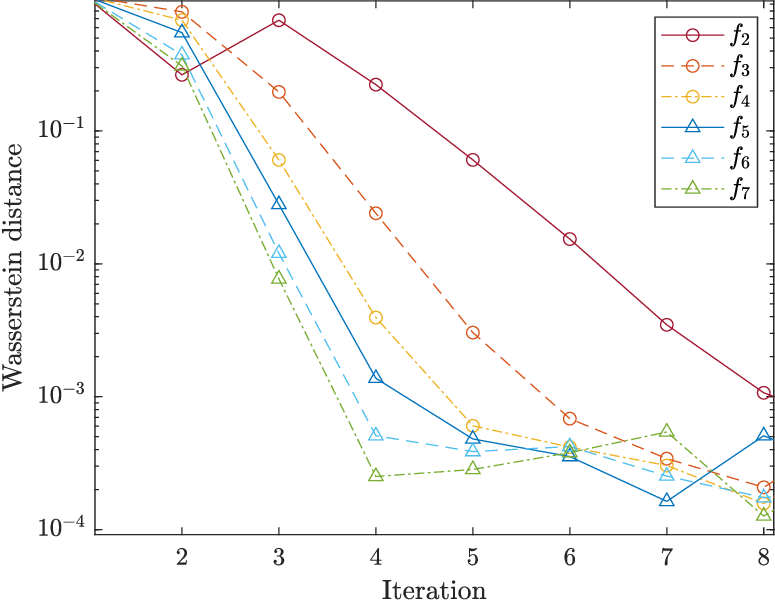}
        \caption{} \label{fig:convergence_rate}
    \end{subfigure}
    \hspace{0.75cm}
    \begin{subfigure}[b]{0.4\textwidth}
        \centering
        \includegraphics[width=\linewidth]{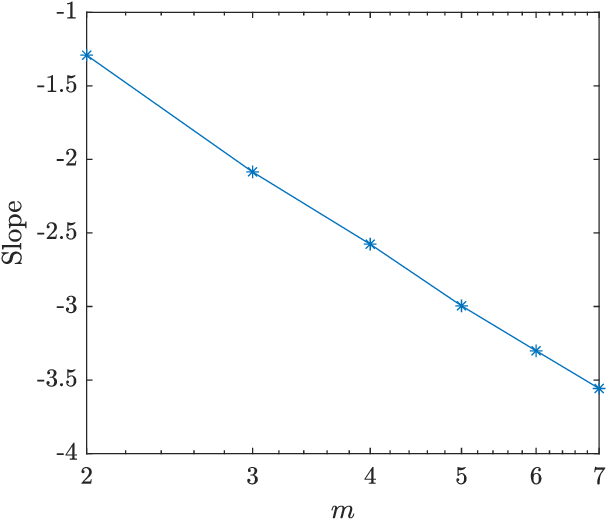}
        \caption{} \label{fig:slopes}
    \end{subfigure}
    \caption{The rate of convergence to the invariant measure of the generalised logistic maps $X_n=f_m^n(X_0)$. Initial values are drawn from $10^8$ independent realisations of the distribution $\Gamma(1,1)-2$ and the Wasserstein distance is used to compare with the invariant measure. (a) shows the convergence for the first eight iterations for different degrees $m=2,3,\dots 7$ of generalised logistic map. The $y$-axis has a logarithmic scale and each curve has a linear region (corresponding to exponential convergence). (b) shows the rate of exponential convergence (\emph{i.e.} the slope of the linear region in (a)) as a function of the degree $m$. This relationship agrees with the logarithmic behaviour predicted by Theorem~\ref{thm:generalised_logistic_convergence}, corroborated by the Lyapunov exponents $\lambda_m=\log m$ in Theorem~\ref{thm:lyapunov}.}
\end{figure}

\section{Concluding remarks}

We have explored the approach of studying nonlinear recursion relations by constructing an associated sequence of periodic operators, which are related to the recursion relation via their discriminants. We demonstrated this approach, which is well established for niche examples of discrete-time dynamical systems \cite{kohmoto1983localization, davies2024vanhove}, by developing it for the canonical examples of the logistic map and the Chebyshev polynomials of the first kind. In these cases, the sequence can be constructed by artificially inflating unit cells for a single periodic potential (thereby folding the spectral bands). This approach provides an immediate candidate for the invariant measure and facilitates a direct proof of convergence in $L^1$ space, along with an explicit estimate of the convergence rates.

The framework developed in this work offers a new way to think about the statistics of chaotic discrete-time dynamical systems. It generates new intuition by \emph{e.g.} linking the singularities in the invariant measures to Van Hove singularities in the densities of states of periodic elliptic operators. It also provides a systematic approach for generating candidate invariant measures and explicit formulas. While these are both well known for the case of the logistic map, this can be generalised to settings where the invariant measure is less obvious (as in \cite{davies2024vanhove}) or can be used to generate novel formulas (as in Section~\ref{sec:Mathieu}).

Within the realm of the logistic map (and generalised logistic maps), there are several open questions which this framework could be used to address in future. When $r<4$, there are no known analytic expressions for the absolutely continuous invariant measures, however they have been observed to have the same singularities \cite{hall2018properties, diaz-ruelas2021logistic}, suggesting a similar underlying link to Van Hove singularities. As a possible extension of this, random logistic maps (where the parameter $r$ is allowed to vary randomly) have been studied and shown to admit a well-defined invariant measure \cite{bhattacharya2004stability, ayers2024noisy, athreya2000random}.

\appendix
\section{Proof of convergence using ergodic theory} \label{app:ergodic}

In this appendix, we apply results from ergodic theory to prove convergence to the absolutely continuous invariant measure for the chaotic logistic map, as stated in Theorem~\ref{thm:ergodic}. For the sake of notation, the most straightforward formulation of this argument is to consider the multiple-tent map $g_m$ defined in \eqref{eq:gm}, which was shown to be conjugate to the generalised logistic maps in Section~\ref{sec:conjugacy}. Since $g_m$ is a piecewise linear one-dimensional recursion, we can use the results presented in \cite{boyarsky1997laws}. In particular, we make the following definitions
\begin{definition} \label{def:mixing}
    Consider a probability measure space $(X,\mathcal{B},\mu)$ and a measurable recursion relation $\tau:X\rightarrow X$ such that $\mu$ is an invariant measure of $\tau$ (or $\mu$ is $\tau$-invariant). Then, $(\tau, \mu)$ is said to be
    \begin{enumerate}[(i)]
    \item \emph{strongly mixing} if for all $A$, $B$ in the $\sigma$-algebra $\mathcal{B}$,
    \begin{equation}
        \mu(\tau^{-n}A\cap B) \rightarrow \mu(A)\mu(B)\quad\text{as }n\rightarrow \infty;
    \end{equation}
    \item \emph{weakly mixing} if for all $A$, $B$ in the $\sigma$-algebra $\mathcal{B}$,
    \begin{equation}
        \frac{1}{n}\sum_{i=0}^{n-1}|\mu(\tau^{-1}A\cap B) - \mu(A)\mu(B)| \rightarrow 0 \quad\text{as } n \rightarrow \infty.
    \end{equation}
    \end{enumerate}
\end{definition}

Strongly mixing implies weakly mixing, which implies ergodicity: $\tau$ is ergodic if for any $B\in\mathcal{B}$ such that $\tau^{-1}B=B$, $\mu(B)=0$ or $\mu(X\setminus B)=0$.

We will demonstrate that the multiple-tent maps $g_m$ are strongly mixing and use this to prove convergence based on the following two results. The first theorem establishes weak convergence for any initial distribution that is absolutely continuous with respect to $\mu$, as in Theorem \ref{thm:ergodic}.

\begin{theorem} \label{thm:weak_convergence} \cite[Proposition~4.2.12]{boyarsky1997laws}
    Let $\tau:I\rightarrow I$ be strongly mixing on the probability measure space $(I,\mathcal{B}, \mu)$ and $P_\tau:I\rightarrow I$ be its Frobenius-Perron operator. Let $p\in L^1(I)$ be a probability density function. Then on any $B\in\mathcal{B}$, it holds that
    \begin{equation*}
        \int_B P_{\tau^n} p\,d\mu \rightarrow \mu(B)\quad\text{as $n\rightarrow\infty$.}
    \end{equation*}
\end{theorem}

To show further the strong convergence (with an exponential convergence rate), we need to consider $BV(I)$ the space of bounded variation functions on $I$, with the norm
\begin{equation*}
    \lVert p \rVert_{BV} = \lVert p \rVert_{ L^1} + \inf_{\Tilde{p}=p\text{ a.e.}}V_I(\Tilde{p}),
\end{equation*}
where $V_I(p)$ is the variation of $p$ on $I$. It should be noted that the result does not apply directly to the logistic map $F$, since it has an invariant density \eqref{eq:logistic_invariant_measure} of unbounded variation.

\begin{theorem} \cite[Theorem~8.3.1]{boyarsky1997laws} \label{thm:strong_convergence}
    Let $\tau:I\rightarrow I$ be piecewise monotonic and $\mu = p\cdot L$ be an invariant measure of $\tau$, that is absolutely continuous with respect to the Lebesgue measure $L$. If $(\tau,\mu)$ is weakly mixing and $\mu$ is the unique absolutely continuous invariant measure of $\tau$, then there exists a positive constant $c$ and a constant $0<r<1$ such that for any $q\in BV(I)$,
    \begin{equation}
        \left\lVert P_{\tau}^n q-\left(\int_I q\,d\mu\right)\,p\right\rVert_{BV} \leq cr^n\lVert q \rVert_{BV}.
    \end{equation}
\end{theorem}

The multiple-tent map $g_m$ on $I=[0,1]$ is clearly piecewise monotonic and it has the Lebesgue measure $L$ as an invariant measure for all $m\in\mathbb{N}$. This is also the unique absolutely continuous invariant measure, whose existence is guaranteed by \cite[Theorem~4]{lasota1982invariant}. 

We will show that $g_m$ is strongly mixing, such that both Theorem~\ref{thm:weak_convergence} \&~\ref{thm:strong_convergence} apply. To show this, we will first establish a result about the action of $g_m$ when expressed in $m$-ary expansions.

\begin{lemma} \label{lm:gmshifting}
    Fix some $m\in\mathbb{N}$ and write both $x\in[0,1]$ and $g_m(x)$ in their $m$-ary expansions
    \begin{equation*}
        x = \sum_{i=1}^{\infty}\frac{a_i}{m^i}\,, \quad g_m(x) = \sum_{i=1}^{\infty}\frac{b_i}{m^i},
    \end{equation*}
    where $a_i, b_i\in\{0,1,2,\dots,m-1\}$ for all $i\in\mathbb{N}$. Then, applying $g_m$ has the effect of shifting the indices to the left once and flipping them when the first index is odd:
    \begin{equation}
        \begin{cases}
            b_i = a_{i+1} \quad\forall i\in\mathbb{N} & \text{when $a_1$ is even} \\
            b_i = m-1-a_{i+1}\quad\forall i\in\mathbb{N} & \text{when $a_1$ is odd.}
        \end{cases}
    \end{equation}
\end{lemma}
\begin{proof}
    Recall the definition for $g_m$ \eqref{eq:gm} which is divided into two cases. When $x\in[\frac{2j}{m},\frac{2j+1}{m}]$ for $j=0,1,2,\dots,\lfloor\frac{m-1}{2}\rfloor$, the first index $a_1 = 2j$ is even and
    \begin{align*}
        \sum_{i=1}^{\infty}\frac{b_i}{m^i} &= g_m(x) \\
        &= m\times\sum_{i=1}^{\infty}\frac{a_i}{m^i} - 2j \\
        &= \sum_{i=1}^{\infty}\frac{a_{i+1}}{m^i} + m\times\frac{2j}{m} - 2j \\
        &= \sum_{i=1}^{\infty}\frac{a_{i+1}}{m^i},
    \end{align*}
    which implies $b_i = a_{i+1}$.

    Conversely, when $x\in[\frac{2j-1}{m},\frac{2j}{m}]$ for $j=1,2,\dots,\lfloor\frac{m}{2}\rfloor$, the first index $a_1 = 2j-1$ is odd and
    \begin{align*}
        \sum_{i=1}^{\infty}\frac{b_i}{m^i} &= g_m(x) \\
        &= -m\times\sum_{i=1}^{\infty}\frac{a_i}{m^i} + 2j \\
        &= -\sum_{i=1}^{\infty}\frac{a_{i+1}}{m^i} - m\times\frac{2j-1}{m} + 2j \\
        &= 1-\sum_{i=1}^{\infty}\frac{a_{i+1}}{m^i} \\
        &= \sum_{i=1}^{\infty}\frac{m-1}{m^i} - \sum_{i=1}^{\infty}\frac{a_{i+1}}{m^i}\\
        &= \sum_{i=1}^{\infty}\frac{m-1-a_{i+1}}{m^i},
    \end{align*}
    which implies $b_i = m-1-a_{i+1}$.
\end{proof}

We are now ready to show that $g_m$ is strongly mixing.

\begin{theorem}
    The multiple-tent map $g_m$ is strongly mixing for each $m\in\mathbb{N}$.
\end{theorem}
\begin{proof}
    The proof is a generalisation of \cite[Solution 3.4.5]{boyarsky1997laws}, where the authors proved this result for $m=2$. Consider dividing $[0,1]$ into $m$-adic intervals:
    \begin{equation}
    \mathcal{B}_m = \left\{\left[\frac{a}{m^i},\frac{b}{m^i}\right) \ : \ a,b,i\in\mathbb{N} \text{ and } 0\leq a<b\leq m^i\right\}.
    \end{equation}
    Note that whether or not some $x\in[0,1]$ falls in an interval $[a/m^i, b/m^i)\in\mathcal{B}_m$ is determined by the first $i$ indices of its $m$-ary expansion.

    By Lemma \ref{lm:gmshifting}, applying $g_m$ shifts the indices and flips when applicable. Therefore, whether $g_m(x)$ is in a given interval $A:=[a/m^i, b/m^i)\in\mathcal{B}_m$ is determined by the 2\textsuperscript{nd} to the $(i+1)$\textsuperscript{th} indices of the $m$-ary expansion of $x$. By induction, whether $g_m^n(x)$ is in $A$, or equivalently $x\in g_m^{-n}(A)$, is determined by the $(n+1)$\textsuperscript{th} to the $(n+i)$\textsuperscript{th} indices. This observation applies to any $A\in\mathcal{B}_m$.

    Consider any $A,B\in\mathcal{B}_m$. For $n$ large enough, the set of indices that determines $g_m^{-n}(A)$ will be disjoint from the set of indices that determines $B$. That is, for large enough $n$,
    \begin{equation*}
        L(g_m^{-n}(A)\cap B) = L(g_m^{-n}(A))\cdot L(B),
    \end{equation*}
    where $L$ is the Lebesgue measure (which is $g_m$-invariant). Hence, for any $A,B \in\mathcal{B}_m$, 
    \begin{equation}
    \lim_{n\rightarrow\infty}L(g_m^{-n}(A)\cap B) = \lim_{n\rightarrow\infty} L(g_m^{-n}(A))\cdot L(B).
    \end{equation}
    Finally, applying the $\pi$-$\lambda$ system lemma twice, this applies to any $A,B$ in the Borel $\sigma$-algebra $\mathcal{B}$, since $\mathcal{B}_m$ is a $\pi$-system that generates $\mathcal{B}$.
\end{proof}

Now that we have established that $g_m$ is strongly mixing, we may use Theorem \ref{thm:weak_convergence} and Theorem \ref{thm:strong_convergence} to show convergence of absolutely continuous measures under iterations of the Frobenius-Perron operator $P_{g_m}$. Theorem \ref{thm:weak_convergence} applies to measures with any probability density functions, but it only proves weak convergence. Theorem \ref{thm:strong_convergence} shows strong convergence at an exponential rate $O(r^n)$ for density functions with bounded variation, but it doesn't give the explicit value for $r$, which we managed to find by hand in Section \ref{sec:convergence}.

\section{Calculation of $I(\pm1)$} \label{app:integrals}

We can use our knowledge of the average Lyapunov exponent from Theorem~\ref{thm:lyapunov}, along with the invariance of Lyapunov exponent under conjugacy, to prove that $I(a)=0$ for two more values $a=\pm1$.

\begin{theorem}
    \begin{equation*}
        \int_{-\pi/2}^{\pi/2} \log|2\sin y-1|\,dy = \int_{-\pi/2}^{\pi/2} \log|2\sin y+1|\,dy = 0.
    \end{equation*}
\end{theorem}
\begin{proof}
    Note that $I(a)=I(-a)$ for all $a\in\mathbb{R}$ since
    \begin{equation*}
        \begin{split}
            I(-a)&=\frac{1}{\pi}\int_{-\pi/2}^{\pi/2} \log|2\sin y+a|\,dy\\
            &=\frac{1}{\pi}\int_{\pi/2}^{-\pi/2} -\log|2\sin (-y)+a|\,dy\\
            &=\frac{1}{\pi}\int_{-\pi/2}^{\pi/2} \log|2\sin y-a|\,dy = I(a).
        \end{split}
    \end{equation*}
    
    When $m=3$,
    \begin{equation*}
        \lambda_3 = \log3+I(-1)+I(1).
    \end{equation*}
    Since $\lambda_3 = \log 3$, this implies that $I(-1)=I(1)=0$.
\end{proof}

\section*{Acknowledgements}
B.D. was supported by the EPSRC under grant number EP/X027422/1. Y.X. was supported by a Mary Lister McCammon Fellowship at Imperial College London, funded by the EPSRC CDT for Collaborative Computational Modelling at the Interface (grant number EP/Y034767/1).

\section*{Data availability}

No datasets were generated or analysed during the current study.

\bibliographystyle{abbrv}
\bibliography{references}{}

\begin{thebibliography}{10}

\bibitem{athreya2000random}
K.~B. Athreya and J.~Dai.
\newblock Random logistic maps. {I}.
\newblock {\em J. Theor. Probab.}, 13:595--608, 2000.

\bibitem{avila2009ten}
A.~Avila and S.~Jitomirskaya.
\newblock The ten martini problem.
\newblock {\em Ann. Math.}, 170(1):303--342, 2009.

\bibitem{ayers2024noisy}
K.~Ayers and A.~Radunskaya.
\newblock Noisy fixed points: Stability of the invariant distribution of the random logistic map.
\newblock {\em arXiv preprint arXiv:2403.13116}, 2024.

\bibitem{bhattacharya2004stability}
R.~Bhattacharya and M.~Majumdar.
\newblock Stability in distribution of randomly perturbed quadratic maps as {M}arkov processes.
\newblock {\em Ann. Appl. Probab.}, pages 1802--1809, 2004.

\bibitem{boyarsky1997laws}
A.~Boyarsky and P.~Gora.
\newblock {\em Laws of Chaos: Invariant Measures and Dynamical Systems in One Dimension}.
\newblock Springer Science \& Business Media, New York, 2012.

\bibitem{boyarsky2001invariant}
A.~Boyarsky and P.~Góra.
\newblock Invariant measures for {C}hebyshev maps.
\newblock {\em Int. J. Stochastic Anal.}, 14(3):326414, 2001.

\bibitem{bruin2010existence}
H.~Bruin, M.~Demers, and I.~Melbourne.
\newblock Existence and convergence properties of physical measures for certain dynamical systems with holes.
\newblock {\em Ergodic Theory and Dynamical Systems}, 30(3):687–728, 2010.

\bibitem{collet2004liapunov}
P.~Collet and J.-P. Eckmann.
\newblock Liapunov multipliers and decay of correlations in dynamical systems.
\newblock {\em Journal of Statistical Physics}, 115(1/2):217–254, Apr 2004.

\bibitem{davies2024vanhove}
B.~Davies.
\newblock Van {H}ove singularities in the density of states of a chaotic dynamical system.
\newblock {\em arXiv preprint arXiv:2404.12073}, 2024.

\bibitem{davies2022symmetry}
B.~Davies and R.~V. Craster.
\newblock Symmetry-induced quasicrystalline waveguides.
\newblock {\em Wave Motion}, 115:103068, 2022.

\bibitem{davies2023super}
B.~Davies and L.~Morini.
\newblock Super band gaps and periodic approximants of generalised {F}ibonacci tilings.
\newblock {\em Proc. R. Soc. A}, 480(2285):20230663, 2024.

\bibitem{diaz-ruelas2021logistic}
A.~Diaz-Ruelas, F.~Baldovin, and A.~Robledo.
\newblock {Logistic map trajectory distributions: Renormalization-group, entropy, and criticality at the transition to chaos}.
\newblock {\em Chaos}, 31(3):033112, 03 2021.

\bibitem{eichhorn2001lyapunov}
R.~Eichhorn, S.~J. Linz, and P.~Hänggi.
\newblock Transformation invariance of {L}yapunov exponents.
\newblock {\em Chaos, Soliton. Fract.}, 12(8):1377--1383, 2001.

\bibitem{elwahbi2004recursive}
B.~El~Wahbi, M.~Rachidi, and E.~Zerouali.
\newblock Recursive relations, {J}acobi matrices, moment problems and continued fractions.
\newblock {\em Pac. J. Math.}, 216(1):39–50, Sep 2004.

\bibitem{gumbs1988dynamical}
G.~Gumbs and M.~K. Ali.
\newblock Dynamical maps, {Cantor} spectra, and localization for {Fibonacci} and related quasiperiodic lattices.
\newblock {\em Phys. Rev. Lett.}, 60(11):1081, 1988.

\bibitem{hall2018properties}
P.~Hall and R.~C.~L. Wolff.
\newblock Properties of invariant distributions and {L}yapunov exponents for chaotic logistic maps.
\newblock {\em J. R. Stat. Soc. B}, 57(2):439--452, 12 2018.

\bibitem{hilborn2000quantifying}
R.~C. Hilborn.
\newblock {Quantifying Chaos}.
\newblock In {\em {Chaos and Nonlinear Dynamics: An Introduction for Scientists and Engineers}}. Oxford University Press, 09 2000.

\bibitem{jagannathan2021fibonacci}
A.~Jagannathan.
\newblock The {F}ibonacci quasicrystal: Case study of hidden dimensions and multifractality.
\newblock {\em Rev. Mod. Phys.}, 93(4):045001, 2021.

\bibitem{kohmoto1983localization}
M.~Kohmoto, L.~P. Kadanoff, and C.~Tang.
\newblock Localization problem in one dimension: Mapping and escape.
\newblock {\em Phys. Rev. Lett.}, 50(23):1870, 1983.

\bibitem{kolavr1993new}
M.~Kol{\'a}{\v{r}}.
\newblock New class of one-dimensional quasicrystals.
\newblock {\em Phys. Rev. B}, 47(9):5489, 1993.

\bibitem{kuchment2016overview}
P.~Kuchment.
\newblock An overview of periodic elliptic operators.
\newblock {\em Bull. Am. Math. Soc.}, 53(3):343--414, 2016.

\bibitem{kuznetsov2016invariance}
N.~V. Kuznetsov, T.~A. Alexeeva, and G.~A. Leonov.
\newblock Invariance of {L}yapunov exponents and {L}yapunov dimension for regular and irregular linearizations.
\newblock {\em Nonlinear Dyn.}, 85(1):195–201, Feb 2016.

\bibitem{lasota1973existence}
A.~Lasota and J.~A. Yorke.
\newblock On the existence of invariant measures for piecewise monotonic transformations.
\newblock {\em Transactions of the American Mathematical Society}, 186:481--488, 1973.

\bibitem{lasota1982invariant}
A.~Lasota and J.~A. Yorke.
\newblock Exact dynamical systems and the {F}robenius-{P}erron operator.
\newblock {\em Trans. Am. Math. Soc.}, 273(1):375–384, 1982.

\bibitem{butkovskii2008mixing}
M.~Y. Logunov and O.~Y. Butkovskii.
\newblock Mixing and lyapunov exponents of chaotic systems.
\newblock {\em Technical Physics}, 53(8):959–965, Aug 2008.

\bibitem{lorenz1964problem}
E.~N. Lorenz.
\newblock The problem of deducing the climate from the governing equations.
\newblock {\em Tellus}, 16(1):1--11, 1964.

\bibitem{lyubich2000quadratic}
M.~Lyubich.
\newblock The quadratic family as a qualitatively solvable model of chaos.
\newblock {\em Not. Am. Math. Soc.}, 47:1042--1052, 2000.

\bibitem{mandelbrot2004fractals}
B.~B. Mandelbrot.
\newblock {\em Fractals and Chaos: The Mandelbrot Set and Beyond}.
\newblock Springer, New York, 2004.

\bibitem{may1976simple}
R.~M. May.
\newblock Simple mathematical models with very complicated dynamics.
\newblock {\em Nature}, 261(5560):459--467, 1976.

\bibitem{morini2018waves}
L.~Morini and M.~Gei.
\newblock Waves in one-dimensional quasicrystalline structures: dynamical trace mapping, scaling and self-similarity of the spectrum.
\newblock {\em J. Mech. Phys. Solids}, 119:83--103, 2018.

\bibitem{pahade2015trace}
J.~Pahade and M.~Jha.
\newblock Trace of positive integer power of real 2*2 matrices.
\newblock {\em Adv. Lin. Alg. Mat. Theory}, 05(04):150–155, 2015.

\bibitem{peitgen1986beauty}
H.-O. Peitgen and P.~H. Richter.
\newblock {\em The Beauty of Fractals: Images of Complex Dynamical Systems}.
\newblock Springer Science \& Business Media, 1986.

\bibitem{rychlik1983bounded}
M.~Rychlik.
\newblock Bounded variation and invariant measures.
\newblock {\em Studia Mathematica}, 76(1):69--80, 1983.

\bibitem{simon1982almost}
B.~Simon.
\newblock Almost periodic {Schr{\"o}dinger} operators: a review.
\newblock {\em Adv. Appl. Math.}, 3(4):463--490, 1982.

\bibitem{stein2005real}
E.~M. Stein and R.~Shakarchi.
\newblock {\em Real Analysis: Measure Theory, Integration, and Hilbert Spaces}.
\newblock Princeton Univ. Press, 2005.

\bibitem{surace1990schrodinger}
S.~Surace.
\newblock The {S}chr{\"o}dinger equation with a quasi-periodic potential.
\newblock {\em Trans. Amer. Math. Soc.}, 320(1):321--370, 1990.

\bibitem{tsybakov2009introduction}
A.~B. Tsybakov.
\newblock {\em Introduction to Nonparametric Estimation}.
\newblock Springer-Verlag New York: Springer e-books, 2009.

\bibitem{van1953occurrence}
L.~Van~Hove.
\newblock The occurrence of singularities in the elastic frequency distribution of a crystal.
\newblock {\em Phys. Rev.}, 89(6):1189, 1953.

\bibitem{wolf2009mathieu}
G.~Wolf.
\newblock Mathieu functions and {H}ill’s equation.
\newblock {\em NIST Handbook of Mathematical Functions, edited by FWJ Oliver, DW Lozier, RF Boisvert, and CW Clark (Cambridge University Press, Cambridge, 2010)}, pages 651--681, 2009.

\end{thebibliography}
\end{document}